\documentclass[11pt]{amsart}

\usepackage{amssymb,amsfonts,mathtools}
\usepackage[all,arc]{xy}
\usepackage{enumerate}
\usepackage{mathrsfs}
\usepackage{amsmath}
\usepackage{amsthm}
\usepackage{tikz}
\usepackage{tikz-cd}
\usepackage[utf8]{inputenc}
\usepackage{lipsum}
\usepackage{extpfeil}
\usepackage{colonequals}
\usepackage{graphicx}
\usepackage{subfig}
\usepackage{enumitem}
\usepackage{array}
\usepackage{xcolor}
\colorlet{RED}{red}
\usepackage{bm, bbm}

\newtheorem{thm}{Theorem}[section]
\newtheorem{conj}[thm]{Conjecture}
\newtheorem*{conj*}{Conjecture}
\newtheorem{cor}[thm]{Corollary}
\newtheorem{prop}[thm]{Proposition}
\newtheorem{lem}[thm]{Lemma}

\theoremstyle{definition}

\theoremstyle{remark}

\DeclareMathOperator{\ZZ}{\mathbb{Z}}
\DeclareMathOperator{\QQ}{\mathbb{Q}}
\DeclareMathOperator{\NN}{\mathbb{N}}

\DeclareMathOperator{\arm}{\mathrm{arm}}
\DeclareMathOperator{\leg}{\mathrm{leg}}
\DeclareMathOperator{\coarm}{\mathrm{coarm}}
\DeclareMathOperator{\coleg}{\mathrm{coleg}}
\DeclareMathOperator{\Real}{\operatorname{Re}}

\usepackage{tikz-cd}
\usepackage{float}
\usepackage[margin=1in]{geometry}
\usepackage{array}
\usepackage{comment}
\usepackage[symbol]{footmisc}


\title[Hook length biases for self-conjugate and distinct odd parts partitions]{Hook length biases for self-conjugate partitions and partitions with distinct odd parts}

\author[C. Cossaboom]{Catherine H. Cossaboom}
\address{Department of Mathematics, University of Virginia, 141 Cabell Drive, Charlottesville, VA 22903}
\email{qkb9us@virginia.edu}
\date{}

\begin{document}
\begin{abstract} We establish a hook length bias between self-conjugate partitions and partitions of distinct odd parts, demonstrating that there are more hooks of fixed length $t \geq 2$ among self-conjugate partitions of $n$ than among partitions of distinct odd parts of $n$ for sufficiently large $n$. More precisely, we derive asymptotic formulas for the total number of hooks of fixed length $t$ in both classes. This resolves a conjecture of Ballantine, Burson, Craig, Folsom, and Wen.
\end{abstract}

\maketitle

\section{Introduction}
A partition $\lambda = (\lambda_1, \lambda_2, \dots, \lambda_\ell)$ of an integer $n \geq 0$ is a non-increasing sequence of positive integers  $\lambda_1 \geq \lambda_2 \geq \cdots \geq \lambda_\ell$, which sum to $n$. We say that $\lambda$ has size $n$ and denote this by $|\lambda| = n$, and we call $\ell = \ell(\lambda)$ the length of $\lambda$. Further, we let $\mathcal{P}(n)$ be the set of all partitions of $n$, and we define the partition function $p(n) := |\mathcal{P}(n)|$ to count the number of partitions of $n$.

Given two sets of combinatorial objects that enjoy a bijection, one may naively suppose that their arithmetic statistics are equal, at least asymptotically. In the theory of partitions, restricted classes of partitions offer settings where that initial assumption is definitively false. Despite a natural bijection between two classes of partitions, hook numbers of fixed size can be more frequently found in one class than another. 

Here, we consider \textit{hook numbers} or \textit{hooks} of integer partitions. Hook numbers are often studied due to their representation-theoretic connections, determining dimensions of representations of $S_n$ \cite{Stanley}. To define the hook numbers of $\lambda$, we consider the \textit{Ferrers--Young diagram} of $\lambda$, which comprises $\ell$ rows of left-justified boxes, where the $i$th row contains $\lambda_i$ boxes. The hook number of a box at row $i$ and column $j$ is defined to be $(\lambda_i - j) + (\lambda_j - i) + 1$. In words, it is the length of the $L$-shape formed by the boxes below and to the right of the box, including the box itself. See Figure \ref{fig:hook_numbers}.

\begin{figure}[h]
    \centering

    \resizebox{3cm}{!}{
    \begin{tikzpicture}

        % Define box size and text size
        \def\boxsize{0.8}
        \def\textsize{\large}

        % Row 1
        \draw (0, 0) rectangle (\boxsize, \boxsize) node[midway] {\textsize 7};
        \draw (\boxsize, 0) rectangle (2*\boxsize, \boxsize) node[midway] {\textsize 6};
        \draw (2*\boxsize, 0) rectangle (3*\boxsize, \boxsize) node[midway] {\textsize 4};
        \draw (3*\boxsize, 0) rectangle (4*\boxsize, \boxsize) node[midway] {\textsize 2};
        \draw (4*\boxsize, 0) rectangle (5*\boxsize, \boxsize) node[midway] {\textsize 1};

        % Row 2
        \draw (0, -\boxsize) rectangle (\boxsize, 0) node[midway] {\textsize 4};
        \draw (\boxsize, -\boxsize) rectangle (2*\boxsize, 0) node[midway] {\textsize 3};
        \draw (2*\boxsize, -\boxsize) rectangle (3*\boxsize, 0) node[midway] {\textsize 1};

        % Row 3
        \draw (0, -2*\boxsize) rectangle (\boxsize, -\boxsize) node[midway] {\textsize 2};
        \draw (\boxsize, -2*\boxsize) rectangle (2*\boxsize, -\boxsize) node[midway] {\textsize 1};

    \end{tikzpicture}}
    \caption{Hook numbers for the partition (5,3,2)}
    \label{fig:hook_numbers}
\end{figure}

Over the last decades, deep connections between $q$-series and hook numbers have been established, such as the Nekrasov--Okounkov formula \cite{Nekrasov2006} and Han's generalization \cite{Han}. These formulas have spurred extensive research on hook numbers, as in \cite{AS, BCMO, BOW, GKS, OS}, with special attention to the statistic $n_t(\lambda)$, which counts the number of $t$-hooks in the partition $\lambda$, as in \cite{GOT}. Recent studies have frequently discussed $n_t(\lambda)$ in restricted partitions, as in \cite{AAOS, BBCFW, CDH, COS, SinghBarman}.

In \cite{BBCFW}, Ballantine, Burson, Craig, Folsom, and Wen compare the total number of hooks of fixed length in odd partitions of size $n$, denoted $\mathcal{O}(n)$, to distinct partitions, denoted $\mathcal{D}(n)$, for which Euler \cite[Corollary 1.2]{Andrews} establishes a bijection. Precisely, the authors discuss the partition statistics 
\begin{equation} a_t(n) := \sum_{\lambda \in \mathcal{O}(n)} n_t(\lambda) \quad \text{and} \quad b_t(n) := \sum_{\lambda \in \mathcal{D}(n)} n_t(\lambda). \end{equation}

In \cite{Andrews2}, Andrews proved a conjecture of Beck that states that $b_1(n) \geq a_1(n)$ for all $n \geq 0$, providing the first known example of a \textit{hook length bias.} Because Euler's bijection establishes that \begin{equation*} n |\mathcal{O}(n)| = \sum_{t \geq 1} a_t(n) = \sum_{t \geq 1} b_t(n) = n | \mathcal{D}(n)|, \end{equation*}
it is sensible to ask when the bias switches direction: at what point must $a_t(n) \geq b_t(n)$ for $t \geq 2$? The authors of \cite{BBCFW} pioneer this inquiry, showing that $a_2(n) \geq b_2(n)$ and $a_3(n) \geq b_3(n)$ for large $n$. Further, for $t \geq 2$, they conjecture there exists $N_t$ for which $a_t(n) \geq b_t(n)$ for all $n > N_t$. In \cite{CDH}, Craig, Dawsey, and Han prove this conjecture, demonstrating that such biases exist for all $t \geq 2$.

In this paper, we establish hook length biases for two other restricted classes which possess a natural bijection: partitions with distinct odd parts, denoted $\mathcal{DO}(n)$, and self-conjugate partitions, denoted $\mathcal{SC}(n)$. We study the following partition statistics, choosing notation defined in \cite{BBCFW}:
\begin{equation} a^*_t(n) := \sum_{\lambda \in \mathcal{SC}(n)} n_t(\lambda) \quad \text{and} \quad b^*_t(n) := \sum_{\lambda \in \mathcal{DO}(n)} n_t(\lambda). \end{equation}
Heuristically, hook numbers of distinct odd parts partitions tend to be small or large, while hook numbers of self-conjugate partitions tend to take intermediate values. Ballantine, Burson, Craig, Folsom, and Wen made this notion precise in the following conjectures. Craig, Dawsey, and Han strengthened the second statement.
\begin{conj}[Ballantine--Burson--Craig--Folsom--Wen, Craig--Dawsey--Han]\label{conj} 

Let $t \geq 2$. Then the following are true: 
\begin{enumerate}
\item There exists some integer $N^*_t$ such that $a^*_t(n) \geq b^*_t(n)$ for all $n > N^*_t$.
\item There exists some constant $\gamma_t^* > 1$ such that $a^*_t(n)/b^*_t(n) \rightarrow \gamma_t^*$ as $n \rightarrow \infty$. 
\end{enumerate} 
\end{conj}
\noindent We prove this conjecture. It suffices to prove Conjecture \ref{conj}(2), as Conjecture \ref{conj}(1) follows. 

\begin{thm}\label{theorem1} 
Conjecture \ref{conj} is true. 
\end{thm}

Theorem \ref{theorem1} will follow from the components of Theorem \ref{theorem2}. 

\begin{thm}\label{theorem2} We demonstrate the following.
\begin{enumerate}
\item For all $t \geq 1$, we have \begin{equation*} a_t^*(n) \sim \frac{\sqrt[4]{3}}{2 \pi \sqrt[4]{2} \cdot n^{\frac{1}{4}}} e^{ \pi \sqrt{n / 6 }}. \end{equation*}
\item For all $t \geq 1$, there exists a constant $\beta^*_t \in \QQ(\log(2))$ such that \begin{equation*} b_t^*(n) \sim \beta^*_t \frac{\sqrt[4]{3}}{\pi \sqrt[4]{2} \cdot n^{\frac{1}{4}}} e^{ \pi \sqrt{n / 6 }}, 
\end{equation*}
where $\beta^*_t \in \QQ$ if and only if $t \equiv 0 \pmod{3}$. 
\item For all $t \geq 2$, we have $\beta_t^* < \frac{1}{2}$.
\end{enumerate}
\end{thm}

Theorem \ref{theorem2}(1) and (2) together imply that $a^*_t(n)/b^*_t(n) \rightarrow \frac{1}{2\beta^*_t}$ as $t \rightarrow \infty$. Thus, Theorem \ref{theorem1} follows from Theorem \ref{theorem2}(3). We also prove the following result about $\gamma_t^* = \frac{1}{2 \beta_t^*}$.

\begin{thm}\label{theorem3} As $t \rightarrow \infty$, we have that
\begin{align*} \gamma_t^* &\rightarrow \frac{3}{2 \ln \left( 5/2 \right)} = 1.6370350019... \end{align*}
\end{thm}

Figure \ref{fig:gamma_values} illustrates the convergence of $\gamma_t^*$ for $t \geq 2$ numerically. In fact, we produce an explicit formula for $\gamma_t^*$. The formula is quite involved, so its presentation is postponed until Section \ref{sectionevaluatingbeta}. 

\pagebreak

\begin{figure}[h]
    \centering
    \resizebox{4cm}{!}{
    \begin{tabular}{|c|c|}
        \hline
        t & \(\gamma_t^*\) \\
        \hline
         2 & 1.4426950409... \\
        3 & 2.0000000000... \\
        4 & 1.4426950409... \\
        5 & 1.7601073000... \\
        10 & 1.6259576185... \\
        100 & 1.6369011056... \\
        1000 & 1.6366790000... \\
        10000 & 1.6370349885... \\
        \hline
    \end{tabular}}
    \caption{Values of \(\gamma_t^*\) for various \(t\)}
    \label{fig:gamma_values}
\end{figure}

The paper is structured as follows. In Section \ref{sectiongenfunctions}, we construct the generating function for the sequence $b_t^*(n)$, and we present the generating function for $a_t^*(n)$, which was previously constructed in \cite{AAOS}. In Section \ref{sectiongenfnasympt} and \ref{sectionasymptotics}, we prove Theorem \ref{theorem2}(1) and (2). Section \ref{sectiongenfnasympt} is devoted to asymptotics for the generating functions, and Section \ref{sectionasymptotics} builds on these results to produce asymptotics for $a_t^*(n)$ and $b_t^*(n)$ using Wright's Circle Method. In Section \ref{sectionevaluatingbeta}, we evaluate $\beta^*_t$. In Section \ref{sectionconj}, we provide a proof of Theorem \ref{theorem2}(3). Finally, in Section \ref{sectionlimits}, we provide a proof of Theorem \ref{theorem3}. 

\section*{Acknowledgments}

The author thanks Ken Ono for suggesting this project and for helpful discussions related to this paper. The author is also grateful to William Craig and Kathrin Bringmann for valuable comments. A portion of this paper was written while the author was visiting the Max Planck Institute for Mathematics, whose hospitality she acknowledges. This research was partially supported by the Raven Fellowship and the Ingrassia Family Echols Scholars Research Grant. 

\section{Generating Functions for $a^*_t(n)$ and $b^*_t(n)$}\label{sectiongenfunctions}

In this section, we establish generating functions for $a^*_t(n)$ and $b^*_t(n)$. Define $A^*_t(q)$ and $B^*_t(q)$ as \begin{equation} A^*_t(q) := \sum_{n \geq 1} a^*_t(n) q^n \quad \text{and} \quad B^*_t(q) := \sum_{n \geq 1} b^*_t(n) q^n. \end{equation}

Recall the standard notation for the $q$-Pochhammer symbol and $q$-binomial coefficient: 
\begin{equation*} (x;q)_0 := 1, \end{equation*}
\begin{equation*} (x;q)_n := (1-x)(1-xq) \cdots (1-xq^{n-1}), \end{equation*}
\begin{equation*} (x;q)_\infty := \lim_{n \rightarrow \infty} (x;q)_n, \end{equation*}
\begin{equation*}
    \binom{n}{k}_q := \frac{(q;q)_n}{(q;q)_k(q;q)_{n-k}}.
\end{equation*}

Recent work of Amdeberhan, Andrews, Ono, and Singh computes $A_t^*(q)$. 
\begin{thm}[\protect{\cite[Theorem~2.1]{AAOS}}]\label{aaos}  The following are true as formal power series.
\begin{enumerate}
\item If $t$ is even, we have that 
\begin{equation*} A^*_t(q) = \frac{tq^{2t}}{1-q^{2t}}(-q; q^2)_\infty. \end{equation*} 
\item If $t$ is odd, we have that 
\begin{equation*} A^*_t(q) = \frac{q^{t}(1 + (t-1)q^t + tq^{2t})}{(1-q^{2t})(1+q^t)} (-q; q^2)_\infty. \end{equation*} 
\end{enumerate}
\end{thm}

It remains to produce a formula for $B^*_t(q)$. We follow a method described in \cite{BH} to do so.

\begin{thm}\label{bgenfun} The following identities are true as formal power series. For $t \geq 2$ even, we have 
\begin{align*}
B^*_t(q) &= (-q, q^2)_\infty \sum_{k = 0}^{\lfloor (t-4)/6 \rfloor} q^{t+4k^2+4k}  \binom{\frac{t-2k-2}{2}}{2k+1}_{q^2} \sum_{m \geq 0} \frac{q^{2m(2k+2)}}{(-q^{m+1}, q^2)_{\frac{t-2k}{2}}} \\
&+ (-q, q^2)_\infty \sum_{k=0}^{\lfloor (t-2)/6 \rfloor} q^{t+ 4k^2 + 4k + 1}  \binom{\frac{t-2k-2}{2}}{2k}_{q^2} \sum_{m \geq 0} \frac{q^{2m(2k+1)}}{(-q^{m+2}, q^2)_{\frac{t-2k}{2}}}.
\end{align*}
For $t \geq 1$ odd, we have that 
\begin{align*} 
B^*_t(q) 
&= (-q, q^2)_\infty \sum_{k = 0}^{\lfloor (t-1)/6 \rfloor} q^{t+4k^2-4k}  \binom{\frac{t-2k-1}{2}}{2k}_{q^2} \sum_{m \geq 0} \frac{q^{2m(2k+1)}}{(-q^{m+1}, q^2)_{\frac{t-2k+1}{2}}} \\
&+ (-q, q^2)_\infty \sum_{k=0}^{\lfloor (t-5)/6 \rfloor} q^{t+ 4k^2 + 6k + 3}  \binom{\frac{t-2k-3}{2}}{2k+1}_{q^2} \sum_{m \geq 0} \frac{q^{2m(2k+2)}}{(-q^{m+2}, q^2)_{\frac{t-2k-1}{2}}}.
\end{align*}

\end{thm}

\noindent \textit{Proof.} Let $\lambda$ be a partition. For each box $v \in \lambda$, define the arm length of $v$ to be the number of boxes $x$ such that $x$ lies to the right of $v$. Similarly, we define the leg length (resp. coarm length, coleg length) to be the number of boxes $x$ below $v$ (resp. to the left of $v$, above $v$). We denote these quantities by $\arm(\lambda, v) := j$, $\leg(\lambda, v) := \ell$, $\coarm(\lambda, v) := m$, and $\coleg(\lambda, v) := g$. See Figure \ref{fig:armlegso-on}. 

\vspace{-0.75em}
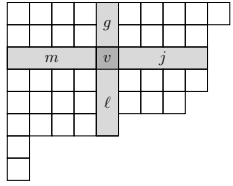
\begin{figure}[h]
    \centering
\resizebox{3cm}{!}{
\begin{tikzpicture}

\fill[gray!30] (4, 2) rectangle (5, 5);
\fill[gray!30] (4, 6) rectangle (5, 8);
\fill[gray!30] (5, 5) rectangle (9, 6);
\fill[gray!60] (4, 5) rectangle (5, 6);
\fill[gray!30] (0, 5) rectangle (4, 6);

% Draw individual squares
\draw (0, 7) rectangle (1, 8);
\draw (1, 7) rectangle (2, 8);
\draw (2, 7) rectangle (3, 8);
\draw (3, 7) rectangle (4, 8);
\draw (5, 7) rectangle (6, 8);
\draw (6, 7) rectangle (7, 8);
\draw (7, 7) rectangle (8, 8);
\draw (8, 7) rectangle (9, 8);
\draw (9, 7) rectangle (10, 8);

\draw (0, 5) rectangle (4, 6);
\draw (4, 5) rectangle (5, 6);
\draw (5, 5) rectangle (9, 6);

\draw (0, 6) rectangle (1, 7);
\draw (1, 6) rectangle (2, 7);
\draw (2, 6) rectangle (3, 7);
\draw (3, 6) rectangle (4, 7);
\draw (5, 6) rectangle (6, 7);
\draw (6, 6) rectangle (7, 7);
\draw (7, 6) rectangle (8, 7);
\draw (8, 6) rectangle (9, 7);

\draw (0, 4) rectangle (1, 5);
\draw (1, 4) rectangle (2, 5);
\draw (2, 4) rectangle (3, 5);
\draw (3, 4) rectangle (4, 5);
\draw (5, 4) rectangle (6, 5);
\draw (6, 4) rectangle (7, 5);
\draw (7, 4) rectangle (8, 5);
\draw (8, 4) rectangle (9, 5);

\draw (0, 3) rectangle (1, 4);
\draw (1, 3) rectangle (2, 4);
\draw (2, 3) rectangle (3, 4);
\draw (3, 3) rectangle (4, 4);
\draw (5, 3) rectangle (6, 4);
\draw (6, 3) rectangle (7,4);
\draw (7,3) rectangle (8, 4); 

\draw (0, 2) rectangle (1, 3);
\draw (1, 2) rectangle (2, 3);
\draw (2, 2) rectangle (3, 3);
\draw (3, 2) rectangle (4, 3);
\draw (4, 2) rectangle (5, 6);

\draw (0,1) rectangle (1,2);
\draw (0,0) rectangle (1,1); 
\draw (4,8) -- (5,8);
\draw (4,2) -- (5,2);
\draw (0,5) -- (0,6);
\draw (9,5) -- (9,6);

% Add labels
\node at (7, 5.5) {\fontsize{20}{22}\selectfont $j$};
\node at (2, 5.5) {\fontsize{20}{22}\selectfont $m$};
\node at (4.5, 7) {\fontsize{20}{22}\selectfont $g$};
\node at (4.5, 5.5) {\fontsize{20}{22}\selectfont $v$};
\node at (4.5, 3.5) {\fontsize{20}{22}\selectfont $\ell$};

\end{tikzpicture}}
    \caption{Arm, coarm, leg, and coleg length of  $\lambda = (10,9,9,9,8,5,1,1)$}
    \label{fig:armlegso-on}
\end{figure}

\vspace{-1em}
Consider the following division of the diagram into four regions, ``cut out" by the arm, coarm, and leg, labeled $A$, $B$, $C$, and $D$, as shown in Figure \ref{fig:ABCD}. In particular, $D$ contains the regions consisting of $v$, the arm, the coarm, and the leg, while $C$ contains the coleg.

\begin{figure}[h]
    \centering

\resizebox{3cm}{!}{
\begin{tikzpicture}

% Draw individual squares
\draw (0, 7) rectangle (1, 8);
\draw (1, 7) rectangle (2, 8);
\draw (2, 7) rectangle (3, 8);
\draw (3, 7) rectangle (4, 8);
\draw (4, 7) rectangle (5, 8);
\draw (5, 7) rectangle (6, 8);
\draw (6, 7) rectangle (7, 8);
\draw (7, 7) rectangle (8, 8);
\draw (8, 7) rectangle (9, 8);
\draw (9, 7) rectangle (10, 8);

\draw (4,2) rectangle (5,5);
\draw (0, 5) rectangle (4, 6);
\draw (4,5) rectangle (5,6);
\draw (5, 5) rectangle (9, 6);
\draw (4,6) rectangle (4,8);

\draw (0, 6) rectangle (1, 7);
\draw (1, 6) rectangle (2, 7);
\draw (2, 6) rectangle (3, 7);
\draw (3, 6) rectangle (4, 7);
\draw (5, 6) rectangle (6, 7);
\draw (6, 6) rectangle (7, 7);
\draw (7, 6) rectangle (8, 7);
\draw (8, 6) rectangle (9, 7);

\draw (0, 4) rectangle (1, 5);
\draw (1, 4) rectangle (2, 5);
\draw (2, 4) rectangle (3, 5);
\draw (3, 4) rectangle (4, 5);
\draw (4, 4) rectangle (5, 5);
\draw (5, 4) rectangle (6, 5);
\draw (6, 4) rectangle (7, 5);
\draw (7, 4) rectangle (8, 5);
\draw (8, 4) rectangle (9, 5);

\draw (0, 3) rectangle (1, 4);
\draw (1, 3) rectangle (2, 4);
\draw (2, 3) rectangle (3, 4);
\draw (3, 3) rectangle (4, 4);
\draw (5, 3) rectangle (6, 4);
\draw (6, 3) rectangle (7,4);
\draw (7,3) rectangle (8, 4); 

\draw (0, 2) rectangle (1, 3);
\draw (1, 2) rectangle (2, 3);
\draw (2, 2) rectangle (3, 3);
\draw (3, 2) rectangle (4, 3);
\draw (4, 2) rectangle (5, 6);

\draw (0,1) rectangle (1,2);
\draw (0,0) rectangle (1,1); 

\fill[gray!20] (0, 0) rectangle (1, 2);
\fill[gray!20] (0, 2) rectangle (5, 5.5);
\fill[gray!20] (0,5) rectangle (9,6);
\fill[gray!20] (5,3) rectangle (8,4.5);
\fill[gray!20] (5,4) rectangle (9,5);
\fill[gray!20] (0,6) rectangle (9,7.5);
\fill[gray!20] (0,7) rectangle (10,8);

\draw[dashed] (4,2) rectangle (5,5);
\draw[dashed] (0, 5) rectangle (4, 6);
\draw[dashed] (4,6) rectangle (5,8);
\draw[dashed] (5,5) -- (5,6);

\draw[thick] (0, 0) rectangle (1, 2);
\draw[thick] (0,2) -- (5,2);
\draw[thick] (5,2) -- (5,5);
\draw[thick] (5,5) -- (9,5);
\draw[thick] (0,6) -- (9,6);
\draw[thick] (5,3) -- (8,3);
\draw[thick] (8,3) -- (8,4);
\draw[thick] (8,4) -- (9,4);
\draw[thick] (9,4) -- (9,7);
\draw[thick] (9,7) -- (10,7);
\draw[thick] (10,7) -- (10,8);
\draw[thick] (9,8) -- (10,8);
\draw[thick] (0,0) -- (0,8);
\draw[thick] (0,8) -- (10,8);

% Add labels
\node at (0.5,1) {\fontsize{20}{22}\selectfont $A$};
\node at (3,4) {\fontsize{20}{22}\selectfont $D$};
\node at (6.75,4) {\fontsize{20}{22}\selectfont $B$};
\node at (6,7) {\fontsize{20}{22}\selectfont $C$};

\end{tikzpicture}}

    \caption{Regions $A$, $B$, $C$, and $D$ of $\lambda = (10,9,9,9,8,5,1,1)$}
    \label{fig:ABCD}
\end{figure}
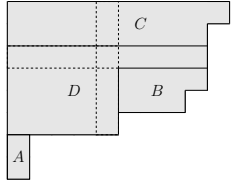

\vspace{-0.5em}
Fix a triple of integers $(j, \ell, m)$, and let $f(j, \ell, m;n)$ denote the number of ordered pairs $(\lambda, v)$ such that $v \in \lambda$, $\lambda \in \mathcal{DO}(n)$, $\arm(\lambda, v) = j $, $\leg(\lambda, v) = \ell$, and $\coarm(\lambda, v) = m$. Consider the generating function $F(j, \ell, m; q) = \sum_n f(j, \ell, m; n) q^n$. We produce a formula for $F(j, \ell, m; q)$, depending on the parity of $m$. Let $F(R; j, \ell, m, q)$ denote the generating function for the number of partitions exhibited by the region $R \in \{A, B, C, D \}$. First, suppose $m = 2m'$ for $m' \in \mathbb{Z}_{\geq 0}$. Since $m+j+1$ is odd, $j = 2j'$ is even with $j' \in \mathbb{Z}_{\geq 0}$. Using routine $q$-series manipulations (see \cite{Andrews}), we find
\begin{align*}
F(A; 2j', \ell, 2m'; q) &= (1+q)(1+ q^3)(1+q^5) \cdots (1+q^{2m'-1}) = (-q; q^2)_{m'}
\end{align*}
\begin{align*}
F(B; 2j',\ell, 2m'; q) &= \binom{j'}{\ell}_{q^2} q^{\ell(\ell - 1)} \\
F(C; 2j', \ell, 2m'; q), &= (1+q^{j+m+3})(1+q^{j+m+5}) \cdots = \frac{(-q; q^2)_\infty}{(-q; q^2)_{j'+m'+1}}, \\
F(D, 2j', \ell, 2m'; q) &= q^{(m+1)(\ell + 1) + j} = q^{(2m'+1)(\ell+1) + 2j'}, \end{align*}
and we obtain the formula
\begin{align*} F(2j', \ell, 2m'; q) &= F(A; 2j', \ell, 2m'; q) F(B; 2j', \ell, 2m';  q) F(C; 2j', \ell, 2m'; q) F(D; 2j', \ell, 2m'; q)
\\
&= q^{(2m'+1)(\ell +1) + 2j' + \ell(\ell -1)} \binom{j'}{\ell}_{q^2} \frac{(-q; q^2)_\infty}{(-q^{2m'+1};q^2)_{j'+1}}. \end{align*}
Similarly, suppose $m = 2m' + 1$ and $j = 2j' + 1$ for $j', m' \in \mathbb{Z}_{\geq 0}$. Analogously to above, we get
\begin{equation*} F(2j'+1, \ell, 2m'+1; q) = q^{(2m'+2)(\ell+1)+2j'+1+\ell^2} \binom{j'}{\ell}_{q^2} \frac{(-q;q^2)_\infty}{(-q^{2m'+3};q^2)_{j'+1}}.
\end{equation*}
\normalsize 

We see that $F(j, \ell, m; q) = 0$ if $j' < \ell$ since then $\binom{j'}{\ell}_{q^2} = 0$. Thus, if $m$ and $j$ are even, $t = \ell + 2j' + 1$ implies that $\ell \leq \frac{t-1}{3}$, where $\ell$ has a different parity than $t$. On the other hand, when $m$ is odd, $t = \ell + 2j' + 2$ implies that $\ell \leq \frac{t-2}{3}$, where $\ell \equiv t \pmod{2}$. Thus, when $t$ is even, we obtain the following. Here, we reindex by $k$, with $\ell = 2k$ and $\ell = 2k+1$, depending on $\ell \pmod{2}$.
\begin{align*} B^*_t(q) &= \sum_{\substack{\ell \leq \lceil t/3 \rceil - 1  \\ \ell \equiv 1 \pmod{2}}} \  \sum_{m' \geq 0} F(t-\ell-1, \ell, 2m'; q) + \sum_{\substack{\ell \leq \lceil (t-1)/3 \rceil - 1 \\ \ell \equiv 0 \mod{2}}} \ \sum_{m' \geq 0} F(t-\ell-1, \ell, 2m'+1;q) \\
&= (-q; q^2)_\infty \sum_{k = 0}^{\lfloor (t-4)/6 \rfloor} q^{t+4k^2+4k}  \binom{\frac{t-2k-2}{2}}{2k+1}_{q^2} \sum_{m' \geq 0} \frac{q^{2m'(2k+2)}}{(-q^{2m'+1}; q^2)_{\frac{t-2k}{2}}} \\
&+ (-q; q^2)_\infty \sum_{k=0}^{\lfloor (t-2)/6 \rfloor} q^{t+ 4k^2 +4k + 1}  \binom{\frac{t-2k-2}{2}}{2k}_{q^2} \sum_{m' \geq 0} \frac{q^{2m'(2k+1)}}{(-q^{2m'+3}; q^2)_{\frac{t-2k}{2}}}.
\end{align*}
\normalsize

Similarly, when $t$ is odd, we get 
\begin{align*} B^*_t(q) 
&= (-q; q^2)_\infty \sum_{k = 0}^{\lfloor (t-1)/6 \rfloor} q^{t+4k^2-4k}  \binom{\frac{t-2k-1}{2}}{2k}_{q^2} \sum_{m' \geq 0} \frac{q^{2m'(2k+1)}}{(-q^{2m'+1}; q^2)_{\frac{t-2k+1}{2}}} \\
&+ (-q; q^2)_\infty \sum_{k=0}^{\lfloor (t-5)/6 \rfloor} q^{t+ 4k^2 + 6k + 3}  \binom{\frac{t-2k-3}{2}}{2k+1}_{q^2} \sum_{m' \geq 0} \frac{q^{2m'(2k+2)}}{(-q^{2m'+3}; q^2)_{\frac{t-2k-1}{2}}}. \qquad \qquad \qquad \qquad \qed
\end{align*}

\section{Asymptotics for $A^*_t(q)$ and $B^*_t(q)$}\label{sectiongenfnasympt}

To prove Theorem \ref{theorem2}(1) and (2), we need strong asymptotic properties for $A^*_t(q)$ and $B^*_t(q)$ with $q = e^{-z}$, as $z \rightarrow 0$ in certain regions. We prove the following.

\begin{prop}\label{Atasymptotics} As $z \rightarrow 0$ with $\Real(z) > 0$, we have that 
\begin{equation*} A^*_t(q) \sim \frac{1}{2} \frac{(-q; q^2)_\infty}{z}. \end{equation*}
\end{prop}

\begin{prop}\label{Btasymptotics} As $z \rightarrow 0$ with $\Real(z) > 0$, we have that \begin{equation*} B_t^*(q) \sim \beta_t^* \frac{(-q;q^2)_\infty}{z}, \end{equation*}
where $\beta_t^* > 0$ is a constant.
\end{prop}

Note that $\beta_t^*$ will be the constant that appears in Theorem \ref{theorem2}(2).

\subsection{Proof of Proposition \ref{Atasymptotics}} 
Theorem \ref{aaos} indicates the Laurent series for $A^*_t(q)/(-q;q^2)_\infty$ is
\begin{equation}\label{laur} \frac{A^*_t(q)}{(-q;q^2)_\infty} = \begin{cases} \frac{1}{2z} - \frac{t}{2} + \frac{t^2 z}{6} - \frac{t^4 z^3}{90} + \frac{t^6 z^5}{945} + O(z^6), \\ \frac{1}{2z} + \left( \frac{1}{4} - \frac{t}{2} \right) + \frac{t^2 z}{6} - \frac{t^4 z^3}{90} + \frac{t^4 z^4}{96} + \frac{t^6 z^5}{945} + O(z^6). \end{cases} \end{equation}
Regardless of the value of $t$, as $z \rightarrow 0$, we have that $\displaystyle{A^*_t(q) \sim \frac{1}{2} \frac{(-q; q^2)_\infty}{z}.}$

\subsection{Proof of Proposition \ref{Btasymptotics}}
In order to understand the asymptotic behavior of $B^*_t(q)/ (-q;q^2)_\infty$, we study the more general functions \begin{equation*} F_{a,b,c}(q) = \sum_{m \geq 0} \frac{q^{am}}{(-q^{2m+b}; q^2)_{c}}.\end{equation*} 
From Theorem \ref{bgenfun}, we have that for even $t$, as $z \rightarrow 0$,
\begin{align*}
B^*_t(q) &\sim (-q; q^2)_\infty \left( \sum_{k = 0}^{\lfloor (t-4)/6 \rfloor} \binom{\frac{t-2k-2}{2}}{2k+1} F_{2(2k+2), 1, \frac{t-2k}{2}}(q) + \sum_{k=0}^{\lfloor (t-2)/6 \rfloor} \binom{\frac{t-2k-2}{2}}{2k} F_{2(2k+1), 3, \frac{t-2k}{2}}(q) \right).
\end{align*} 
If $t$ is odd, as $z \rightarrow 0$, we have the asymptotic formula
\begin{align*} 
B^*_t(q) 
&\sim (-q, q^2)_\infty \left( \sum_{k = 0}^{\lfloor (t-1)/6 \rfloor} \binom{\frac{t-2k-1}{2}}{2k} F_{2(2k+1), 1, \frac{t-2k+1}{2}}(q) + \sum_{k=0}^{\lfloor (t-5)/6 \rfloor} \binom{\frac{t-2k-3}{2}}{2k+1} F_{2(2k+2), 3, \frac{t-2k-1}{2}}(q) \right).
\end{align*}

We use a modification of Euler-Maclaurin summation developed in \cite{BJSM} to produce asymptotics for $F_{a,b,c}(q)$. Let $B_n(x)$ denote the Bernoulli polynomials, let $\widetilde{B}_n(x) := B_n (\{ x \})$, and let $R_\Delta := \{ x + iy \vert \  |y|  \ \leq \Delta x \}$. Further, let $f$ be a holomorphic complex-variabled function in $R_\Delta$ such that $f$ and all of its derivatives decay at infinity \textit{sufficiently}, i.e. at least as fast as $|z|^{1-\varepsilon}$ for some $\varepsilon > 0$. 

\begin{prop}\label{em}{\cite[Proposition~3.1]{CDH}}
For each $N \geq 1$, as $z = x+iy \rightarrow 0$ in $R_\Delta$, we have that
\begin{align*} \sum_{m \geq 0} f((m+1)z) = \frac{1}{z} \int_0^\infty f(x) dx &- \sum_{k \geq 0} \frac{f^{(k)}(0) z^k}{(k+1)!} - \sum_{n=0}^{N-1} \frac{B_{n+1}(0)f^{(n)}(z)}{(n+1)!} z^n \\ &- \frac{(-1)^N z^{N-1}}{N!} \int_{z}^{z \infty} f^{(N)}(w) \widetilde{B}_N \left( \frac{w}{z} - 1 \right) dw, \end{align*}
when $f$ and all its derivatives have sufficient decay at infinity, where the last integral is taken along a path of fixed argument. 
\end{prop}

We apply Proposition \ref{em} to $F_{a,b,c}(q)$. 

\begin{prop} As $z \rightarrow 0$ with $\Real(z) > 0$, we have 
\begin{equation*} F_{a,b,c}(e^{-z}) \sim \frac{1}{z} \int_0^\infty \frac{e^{-ax}}{(1+e^{-2x})^c} dx. \end{equation*}
\end{prop}

\begin{proof}
Let $\Delta > 0$. Let $t, z$ be complex numbers in $R_\Delta$ as in Proposition \ref{em}. Define the functions
\begin{align*} f_{a,b,c}(t;z) &:= \frac{e^{-az}}{(e^{-2z- bt};e^{-2t})_c}, \\
F_{a,b,c}(t;z) &:= \sum_{m \geq 0} f_{a,b,c}(t; mz) = f_{a,b,c}(t; 0) + \sum_{m \geq 0} f_{a,b,c}(t; (m+1)z), \end{align*} 
so that $F_{a,b,c}(z;z) = F_{a,b,c}(e^{-z})$. From Proposition \ref{em} with $t$ fixed, we obtain, for any $N \geq 1$, % \textcolor{red}{check conditions more thoroughly}
\pagebreak
\begin{align*} F_{a,b,c}(t;z) &= f_{a,b,c}(t;0) + \frac{1}{z} \int_0^\infty f_{a,b,c}(t;x) dx - \sum_{m \geq 0} \frac{f_{a,b,c}^{(m)}(t;0) z^m}{(m+1)!} - \sum_{n=0}^{N-1} \frac{B_{n+1}(0) f_{a,b,c}^{(n)}(t;z)}{(n+1)!}z^n \\ &- \frac{(-1)^N z^{N-1}}{N!} \int_z^{z\infty} f_{a,b,c}^{(N)}(t; w) \widetilde{B}_N \left( \frac{w}{z} - 1 \right) dw. \end{align*} 
$f_{a,b,c}$ is holomorphic at $t = 0$, so there are no singularities at $z = 0$, given the identification $t=z$. The only term that contributes to the principal part as $z \rightarrow 0$ is $\frac{1}{z} \int_0^\infty f_{a,b,c}(t;x) dx$, and we find
\begin{equation*} F_{a,b,c}(e^{-z}) \sim \frac{1}{z} \int_0^{\infty} f_{a,b,c}(0;x) dx = \frac{1}{z} \int_0^\infty \frac{e^{-ax}}{(1+e^{-2x})^c} \end{equation*}
since it can be shown analytically that $\lim_{z \rightarrow 0} \int_0^\infty f_{a,b,c}(z;x) dx = \int_0^\infty f_{a,b,c}(0;x) dx.$ \end{proof}

To simplify notation, define $I(a,c) := \int_0^\infty \frac{e^{-ax}}{(1+e^{-2x})^c} dx$. For $t$ even, we define
\small
\begin{equation}\label{betastarteven} \beta_t^* := \sum_{k = 0}^{\lfloor (t-4)/6 \rfloor} \binom{\frac{t-2k-2}{2}}{2k+1} I\left(2(2k+2), \frac{t-2k}{2}\right) + \sum_{k=0}^{\lfloor (t-2)/6 \rfloor}  \binom{\frac{t-2k-2}{2}}{2k} I\left(2(2k+1), \frac{t-2k}{2}\right). \end{equation}
\normalsize
For $t$ odd, we define
\small
\begin{equation}\label{betastartodd} \beta_t^* := \sum_{k = 0}^{\lfloor (t-1)/6 \rfloor} \binom{\frac{t-2k-1}{2}}{2k} I\left(2(2k+1), \frac{t-2k+1}{2}\right) + \sum_{k=0}^{\lfloor (t-5)/6 \rfloor}  \binom{\frac{t-2k-3}{2}}{2k+1} I\left(2(2k+2), \frac{t-2k-1}{2}\right). \end{equation}
\normalsize
Given this value of $\beta^*_t$, we have the desired result. 

% \textcolor{blue}{Simple Laurent series for asymptotics for the generating functions produced by AAOS: For both odd and even $t$, $\tilde{b}^*_t(n) \sim \frac{1}{2} \frac{(-q, q^2)_\infty}{z}$ for $q = e^{-z}$ as $z \rightarrow 0$ in any conical region.}  

% \textcolor{blue}{Following Euler-Maclaurin summation adaptation in Craig, Dawsey, Han in producing asymptotics for generating function for distinct odd parts. I essentially have this; carries through very similarly, have some details to clarify here: $\tilde{a}^*_t(n) \sim \alpha_t \frac{(-q, q^2)_\infty}{z}$ as $z \rightarrow 0$ in any conical region, $\alpha_t$ has expression in terms of linear combination of binomial coefficients and constants coming from simple integrals that can be computed directly.} 

\section{Asymptotics for $a^*_t(n)$ and $b^*_t(n)$}\label{sectionasymptotics}

We use the Ngo and Rhoade's formulation of Wright's Circle Method \cite{NgoRhoades, Wright} to produce asymptotics for $a^*_t(n)$ and $b^*_t(n)$ from those for $A^*_t(q)$ and $B^*_t(q)$ and prove Theorem \ref{theorem2}(1) and (2). 

\subsection{A Variation of Wright's Circle Method} Here, we recall a result of Ngo and Rhoades \cite{NgoRhoades}, which is a modern formulation of Wright's circle method \cite{Wright}. 

In 1971, Wright adapted Hardy and Ramanujan's circle method \cite{HardyRamanujan} to produce asymptotics for the coefficients of a $q$-series $F(q)$ which do not necessarily have a modular transformation law. $F(q)$ need only have a ``main" singularity at $q=1$ and satisfy suitable analytic properties. Given a circle $\mathcal{C}$ of radius less than $1$, we define a \textit{major arc} as the region where $F(q)$ is large, which is $C' = C \cap R_{\Delta}$, and the \textit{minor arc} as $C \setminus C'$. The integral taken over $C'$ constitutes the main term for the coefficients of $F(q)$, while the integral taken over $C \setminus C'$ constitutes the error term. 

Ngo and Rhoades \cite[Proposition 1.8]{NgoRhoades} proved the following, which demonstrates asymptotics for a wide class of functions of the form $L \xi$ where $L$ is asymptotically ``of polynomial size" and $\xi$ grows exponentially, with a primary exponential singularity at $1$. We recall this result below.

\begin{prop}\label{circlemethod}
    Let $N \in \mathbb{N}$ and $\Delta \in \mathbb{R}^+$ be fixed. Suppose that $c(n)$ are integers defined by $\sum_{n \geq 0 } c(n) q^n = L(q) \xi(q)$
    for analytic functions $L, \xi$ within the unit disk satisfying the following hypotheses for $z = x+iy$ with $x > 0$, $0 \leq |y| \leq \pi$: 
    \begin{enumerate}
        \item[(H1)] As $|z| \rightarrow 0$ in the bounded cone $R_\Delta$ as defined in Proposition \ref{em}, we have \\ 
        $\displaystyle{L(e^{-z}) \sim \frac{1}{z} \sum_{k \geq 0} a_k z^k}$
        for $a_k \in \mathbb{C}$, 
        \item[(H2)] As $|z| \rightarrow 0$ in $R_\Delta$, we have 
        $\displaystyle{ \xi(e^{-z}) = K e^{A/z} \left( 1 + O_\theta (e^{-B/z} ) \right) }$
        for $K, A \geq 0$ and $B > A$, 
        \item[(H3)] As $|z| \rightarrow 0$ outside of $R_\Delta$, we have 
        $\displaystyle{L(e^{-z}) \ll_\theta |z|^{-C}}$
        for some $C > 0$, and 
        \item[(H4)] As $|z| \rightarrow 0$ outside of $R_\Delta$, we have 
        $\displaystyle{|\xi(e^{-z})| \ll_\Delta \xi (|e^{-z}|) e^{-\varepsilon/\Real(z)}}$
        for some $\varepsilon > 0$, depending on $\Delta$. 
    \end{enumerate}
    Then as $n \rightarrow \infty$, we have that 
    \begin{equation*} c(n) = Ke^{2\sqrt{An}} n^{-1/4} \left( \sum_{r=0}^{N-1} p_r n^{-r/2} + O \left( n^{-N/2} \right) \right), \end{equation*}
    where $p_r := \sum_{j=0}^r a_j c_{j, r-j}$ with $a_j \in \mathbb{C}$ and $c_{j,r} := \displaystyle{\frac{\left( -\frac{1}{4\sqrt{A}} \right)^r \sqrt{A}^{j - \frac{1}{2}}}{2 \sqrt{\pi}} \frac{\Gamma(j + \frac{1}{2} + r)}{r! \Gamma(j + \frac{1}{2} - r)}}$.
\end{prop}

In this section, we set $q := e^{-z}$ for $z = x+iy$ with $x > 0$, $0 \leq |y| \leq \pi$. Let $\xi(q):= (-q; q^2)_\infty$, $K_t(q) := A^*_t(q)/(-q;q^2)_\infty$, and $L_t(q) := B^*_t(q)/(-q;q^2)_\infty$, so we have $A_t^*(q) = K_t(q) \xi(q)$ and $B_t^*(q) = L_t(q) \xi(q)$. In the following, we bound the major and minor arcs for $\xi(q)$, $K_t(q)$, and $L_t(q)$.

\subsection{Major and Minor Arc Computations for $K_t(q)$}  Proposition \ref{Atasymptotics} implies Lemma \ref{KH1}.

\begin{lem}\label{KH1} For every $t$ and $\Delta > 0$, as $z \rightarrow 0$ in $R_\Delta$, we have
\begin{equation*} K_t(q) \sim \frac{1}{2z}. \end{equation*}
\end{lem}
Lemma \ref{KH3} concerns the region outside of $R_\Delta$.

\begin{lem}\label{KH3} For every $t$ and $A > 0$, as $z \rightarrow 0$ outside $R_\Delta$, we have 
\begin{equation*} |K_t(q)| \ll |z|^{-1}. \end{equation*}
\end{lem}

\noindent \begin{proof} $K_t(q)$ has a convergent Laurent series near $z \rightarrow 0$ for all $t$, as in \eqref{laur}. The term of minimum degree in the Laurent series is $\frac{1}{2z}$ in both cases. The triangle inequality yields the result. \end{proof}

%We apply the triangle inequality and the reverse triangle inequality. Recall that $z = x + iy$ with $x \geq 0$. Now, note that for odd $t$, we have 
%\begin{equation*} |K_t(q)| = \left\vert \frac{t e^{-2tz}}{1 - e^{-2tz}} \right\vert \leq \frac{t e^{-2tx}}{1-e^{-2tx}} = \frac{t}{e^{2tx}- 1} \leq \frac{t}{x} \ll |z|^{-1} \end{equation*}

%For even $t$, we have 
%\begin{align*} |K_t(q)| &= \left\vert \frac{e^{-2tz}(1 + (t-1)e^{-tz} + te^{-2tz})}{(1-e^{-2tz})(1+e^{-tz})} \right\vert \leq \left\vert \frac{e^{-2tz}(1 + (t-1)e^{-tz} + te^{-2tz})}{(1-e^{-2tz})e^{-tz}} \right\vert \\
%&\leq \frac{e^{-2tx}(1 + (t-1)e^{-tx} + te^{-2tx})}{(1-e^{-2tx})e^{-tx}} \leq \frac{(t+1)e^{-tx}}{1-e^{-2tx}} \leq \frac{(t+1)e^{-tx}}{1 - e^{-tx}} \leq \frac{t+1}{e^{tx}-1} \leq \frac{t+1}{x} \ll |z|^{-1}. \end{align*}

\subsection{Major and Minor Arc Computations for $L_t(q)$}

Lemma \ref{LH1} follows from Proposition \ref{Btasymptotics}, which demonstrates that $K_t(q)$ is asymptotic to a rational function for every $t$.

\begin{lem}\label{LH1} For every $t$ and $\Delta > 0$ , as $z \rightarrow 0$ in $R_\Delta$, we have
\begin{equation*} L_t(q) \sim \frac{\beta_t^*}{z}. \end{equation*}
\end{lem}

\noindent We now bound $|L_t(q)|$ outside of $R_\Delta$. This follows from the fact that $L_t^*(q)$ is nearly rational, where we keep track of the obstruction to rationality in Proposition \ref{nearlyrational}. We follow a method of \cite{CDH}.

\begin{prop}\label{nearlyrational} We prove the following. 
\begin{enumerate}
    \item When $t$ is a multiple of $3$, $L_t^*(q)$ is rational in $q$. 
    \item When $t$ is not a multiple of $3$, we can express $L^*_t(q)$ as
\begin{equation*} L^*_t(q) = \widetilde{L}^*_t(q) - q^K \sum_{j=0}^\infty \frac{(-q)^{3j + \alpha}}{1 - q^{2j + \alpha}} \end{equation*}
where $\widetilde{L}^*_t(q)$ is rational in $q$, $K$ depends only on $t$, and $\alpha = \lceil \frac{t}{3} \rceil$.
\end{enumerate}
\end{prop}

\begin{proof} Consider the sum
\begin{equation*} F'_{t,k,a,b,c}(q) = \sum_{m \geq 0} \frac{q^{2m(2k+a)}}{(-q^{2(m+b)+1}, q^2)}_{\frac{t-2k +c}{2}}. \end{equation*}
Note that $a$, $b$, and $c$ play different roles than in $F_{a,b,c}(q)$, as we require both more precision in the powers of $q$ and explicit dependences on $t$ and $k$. 

\noindent Thus, we express $L_t^*(q)$ as
\small
\begin{equation}\label{Lt} L_t^*(q) = \begin{cases} \displaystyle{\sum_{k = 0}^{\lfloor (t-4)/6 \rfloor} q^{t + 4k^2 + 4k} \binom{\frac{t-2k-2}{2}}{2k}_{q^2} F'_{t,k,2, 0, 0}(q) + \sum_{k = 0}^{\lfloor (t-2)/6 \rfloor} q^{t+4k^2+4k+1} \binom{\frac{t-2k-2}{2}}{2k}_{q^2} F'_{t,k,1,1,0}(q)} \text{ for } t \text{ even,} \\
\displaystyle{\sum_{k = 0}^{\lfloor (t-1)/6 \rfloor} q^{t + 4k^2 - 4k} \binom{\frac{t-2k-1}{2}}{2k}_{q^2} F'_{t,k,1, 0, 1}(q)+ \sum_{k = 0}^{\lfloor (t-5)/6 \rfloor} q^{t+k^2+6k+3} \binom{\frac{t-2k-2}{2}}{2k}_{q^2} F'_{t,k,2,1,-1}(q)}  \text{ for } t \text{ odd.} \end{cases} \end{equation}

\normalsize
As in \cite[Ch. 3]{Andrews}, we have the following identity: 
\begin{equation}\label{andrewslike} \frac{1}{(-q^{2m+1}, q^2)_n} = \sum_{j=0}^n \binom{n + j - 1}{j}_{q^2} (-q^{2m+1})^j. \end{equation}
By applying \eqref{andrewslike}, rearranging, and using geometric series formulas, we express $F'_{t,a,b,c}(q)$ as follows:
\begin{align*} &F'_{t,k,a,b,c}(q) = \sum_{m \geq 0} q^{2m(2k+a)} \sum_{j=0}^\infty \binom{\frac{t-2k+c}{2} + j -1}{j}_{q^2} (-1)^j (q^{2(m+b)+1})^{j} \\ &= \sum_{j \geq 0} \binom{\frac{t-2k+c}{2} + j -1}{j}_{q^2} \frac{(-q^{2b+1})^j}{1 - q^{2(j+2k+a)}} = \sum_{j \geq 0} \frac{\Big(1 - q^{2 \left( {\frac{t-2k+c}{2} + j -1} \right)}\Big) \cdots (1 - q^{2(j+1)})}{(q^2; q^2)_{\frac{t-2k+c}{2} -1}} \frac{(-q^{2b+1})^j}{1 - q^{2(j+2k+a)}}. \end{align*}
We now consider a term in the above sum for a fixed value of $j$. We take two cases, based on whether the term $1-q^{2(j+2k+a)}$ cancels with a corresponding term in the numerator. 

\textit{\underline{Case 1:}} Suppose that $k \leq \frac{t - 2a + c-2}{6}$, implying that $j + 2k + a \leq \frac{t-2k+c}{2} + j - 1$. For all uses of $F'_{t,k,a,b,c}(q)$ in \eqref{Lt}, we have that $a \geq 1$, and in turn, that $j + 2k + a \geq j + 1$. Thus, we have that $1 - q^{2(j+2k+a)}$ is cancelled out with a factor in the numerator. We find that
\begin{equation*} \binom{\frac{t-2k+c}{2} + j -1}{j}_{q^2} \frac{(-q^{2b+1})^j}{1 - q^{2(j+2k+a)}} = \frac{1}{(q^2; q^2)_{\frac{t-2k+c}{2} -1}} \sum_{s =b}^{b+\frac{t-2k-c}{2} - 1} P_s(t,k,a,b, c;q^2) q^{(2s+1)j}, \end{equation*}
where $P_s(t,k,a,c;q)$ are polynomials in $q^2$ which are independent of $j$. Summing over all $j$ allows us to write $F'_{t,a,b,c}$ as a finite sum of geometric series, producing
\begin{equation*}
F'_{t,a,b,c}(q) = \frac{1}{(q^2; q^2)_{\frac{t-2k+c}{2} -1}} \sum_{s = b}^{b+\frac{t-2k-c}{2} - 1} \frac{P_s(t,k,a,b,c;q^2)}{1 - q^{2s+1}}.
\end{equation*}
Thus, $F'_{t,k,a,b,c}(q)$ is a finite sum of expressions which are rational in $q$. \\

\textit{\underline{Case 2:}} Suppose that $k > \frac{t-2b+c-2}{6}$, so $j + 2k + a > \frac{t-2k+c}{2}$. Here, $1 - q^{2(j+2k+a)}$ does not cancel out with a term in the numerator. We consider the specific cases of $F'_{t,k,a,b,c}(q)$ that appear in \eqref{Lt}. 

Suppose first that $(a,b,c) = (2,0,0)$, so we have that $k > \frac{t-2a+c-2}{6} = \frac{t-6}{6}$. As in \eqref{Lt}, this sum only appears at $k$ where $k \leq \frac{t-4}{6}$. Since $t$ is even, both conditions are satisfied only when $k = \frac{t-4}{6}$. Similarly, both conditions are satisfied only when $k = \frac{t-2}{6}$ for when $(a,b,c) = (1,1,0)$, $k = \frac{t-1}{6}$ for $(a,b,c) = (1,0,1)$, and $k = \frac{t-5}{6}$ when $(a,b,c) = (2,1,-1)$. Each of these $k$ values only occurs when $t$ is a distinct residue mod $6$ that is not divisible by $3$. Therefore, when $3 \mid t$, $L_t^*(q)$ is rational. The relevant value of $k$ is always equal to $\frac{t-2a+c}{6}$. When $k = \frac{t-2a+c}{6}$, we obtain
\begin{equation*}
\binom{\frac{t-2k+c}{2} + j -1}{j}_{q^2} \frac{(-q^{2b+1})^j}{1 - q^{2(j+2k+a)}} = \frac{\Big(1 - q^{2 \left( {\frac{t+a+c}{3} + j -1} \right)}\Big) \cdots (1 - q^{2(j+1)})}{(q^2; q^2)_{\frac{t-2k+c}{2} -1}} \frac{(-q^{2b+1})^j}{1 - q^{2 \left( j + \frac{t+a+c}{3} \right)} }.
\end{equation*}
We define the expression $\displaystyle{T_j(b,d,x; q) := (-1)^j \frac{x^{b} q^j(1-xq^2)(1-xq^4) \cdots (1-xq^{2d})}{1 - xq^{2(d+1)}}}$, and we get
\begin{align*}
F'_{t,k,a,b,c}(q)
&= \frac{1}{(q^2; q^2)_{\frac{t-2k+c}{2} -1}} \sum_{j \geq 0} T_j \left(b,\frac{t+a+c}{3} - 1, q^{2j}; q \right). \end{align*}
We define $Q(x,q)$ and $R(q)$ as polynomials of $x$ so that 
\[ T_j(b,d,x; q) = (-1)^j \left( Q(x,q) + \frac{R(q)}{1 - xq^{2(d+1)}} \right). \]
Considering $x \rightarrow 0$, we find that $R(q) = -Q(0,q)$. Further, we evaluate
\begin{equation}\label{Rformula} R(q) = x^{b}q^j(1-xq^2)(1-xq^4) \cdots (1-xq^{2d}) \vert_{x = q^{-2(d+1)}} = \frac{q^{j} (-1)^d}{q^{(d+2b)(d+1)}} (q^2; q^2)_{d}. \end{equation}

Letting $d = \frac{t+a+c}{3}-1$, we obtain 
\begin{align*} F'_{t,k,a,b,c}(q) &= \frac{1}{(q^2; q^2)_{d}} \sum_{j \geq 0} (-1)^j \left( Q(q^{2j}, q) + R(q) + \frac{R(q) q^{2j+d+1}}{1 - q^{2j+d+1}} \right) \\
&= \frac{1}{(q^2; q^2)_{d}} \sum_{j \geq 0} (-1)^j \left( Q(q^{2j}, q) - Q(0,q) + \frac{R(q) q^{2j+d+1}}{1 - q^{2j+d+1}} \right).
\end{align*}
Each term of the polynomial $Q(q^{2j}, q) - Q(0,q)$ corresponds to a geometric series in $\sum_{j \geq 0} Q(q^{2j}, q)$. It then suffices to show the remaining terms corresponding to $R(q)$ are rational. By \eqref{Rformula}, we get
\begin{align*} \frac{R(q)}{(q^2; q^2)_{d}} \sum_{j \geq 0} (-1)^j \left(\frac{q^{2j + d+ 1}}{1 - q^{2j+d+1}} \right)
&= \frac{-1}{q^{(d+2b)(d+1)}} \sum_{j \geq 0} \frac{(-q)^{3j + d + 1}}{1 - q^{2j + d + 1}}.
\end{align*}

Suppose $3 \nmid t$. In \eqref{Lt}, the binomial coefficient corresponding to $k= \frac{t-2a+c}{6}$ is $1$. Thus, we have
\[ L_t^*(q) = \widetilde{L}_t^*(q) - q^K \sum_{j=0}^\infty \frac{(-q)^{3j + d + 1}}{1 - q^{2j + d + 1}}  \]
where $\widetilde{L}_t^*(q)$ is rational and $K$ is the difference of the power on the leading $q$-term in the expression (a polynomial in $k$ and $t$) and $(d+2b)(d+1)$. Each of $k, b, d$ is determined as a function of $t$: namely, a linear expression in $t$ determined by the residue of $t$ mod $6$. Determining the value of $\frac{t+a+c}{3}$ for each residue of $t$ mod $6$ then returns the result. \end{proof}

\begin{lem}\label{LH3}
For every $t$, as $z \rightarrow 0$ outside $R_{\Delta}$, we have 
\[ |L_t(q)| \ll |z|^{-C} \]
for some nonnegative constant $C$. 
\end{lem}

\begin{proof} 

Proposition \ref{nearlyrational} proves that $L_t^*(q)$ is the difference of a rational function $\widetilde{L}_t^*(q)$ and the series 
\begin{equation*} 
q^K \sum_{j=0}^\infty \frac{(-q)^{3j + \alpha}}{1 - q^{2j + \alpha}}, 
\end{equation*}
where $K \in \mathbb{Z}$ and $\alpha \in \mathbb{Z}^+$. Since rational functions in $q$ have convergent Laurent expansions near $z = 0$, $\widetilde{L}_t^*(q)$ is $O(|z|^{-C'})$ for some nonnegative constant $C'$. 

Since $|q^j| < 1$ for all $j$, we apply the triangle inequality to find that
\begin{equation}\label{alphaseries}
\Bigg\vert \sum_{j=0}^\infty \frac{(-q)^{3j + \alpha}}{1 - q^{2j + \alpha}} \Bigg\vert \leq \sum_{j=0}^\infty \Big\vert \frac{q^{3j+\alpha}}{1 - q^{2j + \alpha}} \Big\vert \leq  \sum_{j=0}^\infty \Big\vert \frac{q^{2j+\alpha}}{1 - q^{2j + \alpha}} \Big\vert.
\end{equation}
Let the series on the right hand side of this inequality be $\mathcal{S}_\alpha$. Further, let $z = x + iy$ and $d(n)$ be the standard divisor function, i.e. $d(n) = \sum_{d \mid n} 1$. Using the reverse triangle inequality, we find 
\begin{equation*}
    \mathcal{S}_\alpha \leq \sum_{j=1}^\infty \Big\vert \frac{q^{j}}{1 - q^{j}} \Big\vert \leq \sum_{j=1}^\infty \frac{e^{-jx}}{ 1 - e^{-jx}} = \sum_{j = 1}^{\infty} d(j) e^{-jx} \leq \sum_{j=1}^\infty j e^{-jx} = \frac{e^{x}}{(e^{x} - 1)^2}.
\end{equation*}

Outside of the region $R_\Delta$, we have that $\frac{1}{|z|^2} \geq \frac{1}{\delta^2 x^2}.$ Thus, $S_\alpha$ is bounded above by $\frac{\delta^2}{|z|^2}$. Since $q^K$ has a Taylor series expansion for all $K$, we then obtain that $L_t^*(q) \ll |z|^{-C}$ for $C = \max \{ 2, C' \}$.
\end{proof}

\subsection{Major and Minor Arc Computations for $\xi(q)$}

We use the modular transformation law for $\mathcal{P}(q)$, as given in \cite{Apostol} for example, to obtain the modular transformation law for $\xi(q)$:
\begin{equation}\label{modular} \xi(q) = \frac{\mathcal{P}(q^2) \mathcal{P}(q^2)}{\mathcal{P}(q) \mathcal{P}(q^4)} =  \text{exp}\left( \frac{\pi^2}{24z} + \frac{z}{12} \right) \frac{\mathcal{P}(\omega^2)^2}{\mathcal{P}(\omega) \mathcal{P}(\omega^4)}, \end{equation}
where $\omega = e^{-\pi^2/z}$. We now determine the behavior of $\xi$ on the major and minor arcs.

\begin{prop}\label{xiH2} Let $\Delta > 0.62$ be a fixed constant with $\delta = \sqrt{1 + \Delta^2}$. Let $z = x + iy$ be a complex number satisfying $0 \leq |y| \leq \Delta x$ with $0 \leq x < \frac{2 \pi^2}{\delta^2}$. Set $\Psi(z) := \exp\left( \frac{\pi^2}{24z} + \frac{z}{12} \right)$. Then, we have  
\begin{equation*} |\xi(e^{-z}) - \Psi(z) | < 214 \vert \Psi(z) e^{-\pi^2/z} |. \end{equation*}
\end{prop}

\begin{proof} Using Euler's Pentagonal Number Theorem, we rewrite \eqref{modular} as
\small
\begin{align*} \xi(z) &= \left( 1 + \sum_{m \geq 1} (-1)^m \left( \omega^{\frac{m(3m+1)}{2}} + \omega^{\frac{m(3m-1)}{2}} \right) \right) \\ &\cdot \left( 1 + \sum_{m \geq 1} (-1)^m \left( \omega^{2m(3m+1)} + \omega^{2m(3m-1)} \right) \right) \cdot \left(1 + \sum_{m \geq 1} p(m) \omega^{2m} \right)^2. \end{align*}
\normalsize
For notational convenience, we let $X$, $Y$, and $Z$ denote the above expressions, as follows: 
\small
\begin{align*} X &:= \sum_{m \geq 1} (-1)^m \left( \omega^{\frac{m(3m+1)}{2}} + \omega^{\frac{m(3m-1)}{2}} \right), \
Y := \sum_{m \geq 1} (-1)^m \left( \omega^{2m(3m+1)} + \omega^{2m(3m-1)} \right), \
Z := \sum_{m \geq 1} p(m) \omega^{2m}. \end{align*}
We now write
\begin{equation*} \frac{\xi(z) - \Psi(z)}{\Psi(z)} = X(1+Y)^2(1+Z) + (Y^2+2Y)(1+Z) + Z. \end{equation*}
We trivially bound $X$ and $Z$ using geometric series to find
\begin{equation}\label{XZ} |X| < \sum_{m \geq 1} |\omega|^m = |\omega| + \frac{|\omega|^2}{1 - |\omega|}, \qquad |Z| < \frac{|\omega|^4}{1 - |\omega|}. \end{equation}
Let $\delta = \sqrt{1 + \Delta^2}$. Inside the region $R_\Delta$, we have that $\Real(\frac{1}{z}) \geq \frac{1}{\delta^2 x}$. Then, $x < \frac{2 \pi^2}{\delta^2} $ implies
\begin{equation}\label{omegabound} |\omega| = \exp \left(-\Real \left( \frac{\pi^2}{z} \right) \right) \leq \exp \left( -\frac{\pi^2 }{\delta^2 x} \right) < \exp(-1/2) \leq 0.61. \end{equation}
Since $\Delta > 0.62$, we find that $\delta > 1.175$ and $\frac{4 \pi^2}{2.35 \delta} > \frac{2 \pi^2}{\delta}$. Now, applying the proof of \cite[Lemma 3.8]{FayeMisheel}, we bound $Y$ by $23.6 |\omega|^2$. Further, we bound $X$ and $Z$ by applying \eqref{omegabound} to \eqref{XZ}, obtaining
\begin{align*} \Big\vert \frac{\xi(z) - \Psi(z)}{\Psi(z)} \Big\vert &< (|\omega| + 2.54 |\omega|^2)(1 + 23.6|\omega|^2)^2(1 + 2.54|\omega|^4) \\
&+ (556.96|\omega|^4 + 47.2|\omega|^2)(1 + 2.54 |\omega|^4) + 2.54 |\omega|^4 \\ &\leq (|\omega| + 1.55|\omega|) \cdot 95.68 \cdot 1.35 + (126.42|\omega| + 28.80|\omega|) \cdot 1.36 + 0.58|\omega| \leq 214|\omega|. \end{align*}
\end{proof}

\vspace{-1em}
We present an upper bound for $\xi(q)$ outside of $R_\Delta$.

\begin{prop}\label{xiH4} Let $k \geq 4$ be an integer, $\Delta$ a nonnegative constant, and let $\delta = \sqrt{1 + \Delta^2}$. Assume that $z = x + iy$ satisfies $\Delta x \leq |y| \leq \pi$. Then, we have that 
\begin{equation*} |\xi(e^{-z})| \leq C \xi(|e^{-z}|) \cdot e^{-\varepsilon/x} \end{equation*}
where $\varepsilon = -\frac{\pi^2}{24} \left( 1 - \frac{1}{\delta^2} \right) + \frac{5}{4k} > 0$ and $C = k^2 e^{\frac{5 \pi}{24 \delta}}$ depend only on $\delta$ and $k$. In particular, we have that $\displaystyle{\xi (e^{-z}) \ll_{\Delta} \xi (|e^{-z}|) e^{-\varepsilon/x}}$
where $\varepsilon$ can be taken to be arbitrarily close to $-\frac{\pi^2}{24} \left( 1 - \frac{1}{\delta^2} \right)$. 

\end{prop}

\begin{proof}
From \eqref{modular}, we have that 
\begin{align*} | \xi(e^{-z}) | &\leq  \exp \left( \frac{\pi^2}{24x} + \frac{\Real(1/z)}{12} \right) \frac{ \mathcal{P} \left( e^{-2\pi^2 \Real(1/z)} \right)^2}{\left\vert \mathcal{P} \left( e^{-\pi^2/z} \right) \mathcal{P} \left( e^{-4\pi^2/z} \right) \right\vert}.
\end{align*}
Outside the region $R_\Delta$, we have that $\Real \left( \frac{1}{z} \right) \leq \frac{1}{\delta^2 x}$. After a second application of \eqref{modular}, we obtain
\begin{equation*} \exp \left( \frac{\pi^2}{24z} + \frac{\Real(1/z)}{12} \right) \leq \Psi(x) \exp \left( \frac{-\pi^2}{24x} \left( 1 - \frac{1}{\delta^2} \right) \right) = \xi(e^{-x}) \cdot \frac{\mathcal{P}\left(e^{-\pi^2 /x}\right) \mathcal{P}\left(e^{-4 \pi^2 /x}\right)}{\mathcal{P}\left(e^{-2 \pi^2 /x}\right)^2} e^{-\varepsilon / x}. \end{equation*}
Since $\delta > 1$, we have that $\Real(\frac{1}{z}) \leq \frac{1}{\delta^2 x} \leq \frac{1}{x}$, so we find 
\begin{align*} \xi(e^{-z}) &\leq \xi(e^{-x}) e^{-\varepsilon / x} \frac{\mathcal{P}(e^{-\pi^2 /x}) \mathcal{P}(e^{-4 \pi^2 /x})}{\mathcal{P}(e^{-2 \pi^2 /x})^2} \frac{ \mathcal{P} \left( e^{-2\pi^2 \Real(1/z)} \right)^2}{\left\vert \mathcal{P} \left( e^{-\pi^2/z} \right ) \mathcal{P} \left( e^{-4\pi^2/z} \right) \right\vert} \\ &\leq \xi(e^{-x}) e^{-\varepsilon / x} \frac{\mathcal{P}(e^{-\pi^2 /x}) \mathcal{P}(e^{-4 \pi^2 /x})}{\left\vert \mathcal{P} \left( e^{-\pi^2/z} \right ) \mathcal{P} \left( e^{-4\pi^2/z} \right) \right\vert}. \end{align*}
Applying Euler's Pentagonal Number Theorem as in Proposition \ref{xiH2} for $1/|\mathcal{P}(e^{- \pi^2/z})|$ yields
\begin{align*} \frac{1}{\vert \mathcal{P}\left(e^{- \pi^2/z} \right) \vert} &= \sum_{m \geq 1} |e^{- \pi^2/z}|^m \leq \sum_{m=1}^k e^{-m\pi^2 \Real(1/z)} + \frac{e^{-k \pi^2 \Real(1/z)}}{1 - e^{-\pi^2 \Real(1/z)}} \leq k + e^{1/{kx}} \leq k e^{1/{kx}} \end{align*}
for any $k \in \mathbb{N}$. Similarly, we have $1 / |\mathcal{P}(e^{- 4\pi^2/z})| \leq k e^{1/{4kx}}$ for any $k \in \mathbb{N}$. Applying \cite[(2.8)]{FayeMisheel} and using the fact that $x \leq \pi/\delta$, we have that 
\begin{equation*} \mathcal{P}( e^{-\pi^2 / x}) \leq \exp \left( \frac{\pi^2 e^{-\pi^2 /x}}{6 ( 1 - e^{- \pi^2 / x}} ) \right) \leq e^{x/6} \leq e^{\frac{\pi}{6 \delta}}.  \end{equation*}
Similarly, we find $\mathcal{P}( e^{-4\pi^2 / x}) \leq e^{x/24} \leq e^{\frac{\pi}{24 \delta}}.$ Combining these inequalities, we get that 
\begin{equation*}  \xi(e^{-z}) \leq k^2 e^{\frac{5 \pi}{24 \delta}} \xi(e^{-x}) e^{-\varepsilon'/x} \cdot e^{5/4kx} = \xi(e^{-z}) \leq k^2 e^{\frac{5 \pi}{24 \delta}} \xi(e^{-x}) e^{-\varepsilon/x}. \end{equation*}
For $k \geq 4$, we have that $\frac{5}{4 k} \leq \frac{\pi^2}{24}$, which implies that $\varepsilon > 0$. 
\end{proof}

% \begin{equation*} |\xi(e^{-z})| \leq \exp \left( \frac{\pi^2\Real(1/z)}{24} + \frac{z}{12} \right) \mathcal{P} \left( e^{-2\pi^2/z} \right)^2 S(z) S'(z) \end{equation*}
%with 
%\begin{equation*} S(z) = \frac{1}{1 - e^{-\pi^2}}, \quad S'(z) = \frac{1}{1 - e^{-4\pi^2 \Real(1/z)}} \end{equation*}

%Taking the natural logarithm of both sides, we have that 
%\begin{equation*} \log \left( \frac{\xi(e^{-z})}{S(z) S'(z)} \right) \leq \frac{\pi^2 \Real(1/z)}{24} + \frac{x}{12} + 2 \log \mathcal{P}(e^{-2\pi^2 \Real(1/z)}) \end{equation*}

%Applying the bound $\log \mathcal{P}(e^{-\mu}) \leq \frac{\pi^2}{6 \mu}$, we have that 
%\begin{equation*} \log \left( \frac{\xi(e^{-z})}{S(z) S'(z)} \right) \leq \frac{\pi^2}{24 x \delta^2} + \frac{x}{12} + \frac{|z|^2}{6 x} \leq \frac{\pi^2}{24 x} \left( 4 + \frac{1}{\delta^2} \right) + \frac{x}{12} \end{equation*}

%We can bound $S(z)$ and $S'(z)$ as follows:
%\begin{equation*} S(z) = \sum_{a=0}^6 e^{-a \pi^2 \Real(1/z)} + \frac{e^{- 7 \pi^2 \Real(1/z)}}{1 - e^{- \pi^2 \Real(1/z) }} \leq 7 + e^{1/(6 x)} \leq 8 e^{1/(6 x)} \end{equation*}
%\begin{equation*} S'(z) = \sum_{a=0}^{6} e^{-4a \pi^2 \Real(1/z)} + \frac{e^{- 28 \pi^2 \Real(1/z)}}{1 - e^{- 4\pi^2 \Real(1/z) }} \leq 7 + e^{1/(24 x)} \leq 8 e^{1/(24 x)} \end{equation*}

% Will add in a few lines that justify the inequalities. 

\subsection{Proof of Theorem \ref{theorem2}(1) and (2)}
Let $z = x + iy$ with $x > 0$, and fix $\Delta > 0.62$. We verify the hypotheses in Proposition \ref{circlemethod}. $K_t(q)$ satisfies (H1) with $a_0 = \frac{1}{2}$ and $a_k = 0$ for $k > 0$ by Lemma \ref{KH1} and (H3) with $C = 1$ by Lemma \ref{KH3}. Proposition \ref{xiH2} implies that $\xi(q)$ satisfies (H2) with $K = 1$, $A = \pi^2 /24$, and $B = \pi^2$, while $\xi(q)$ satisfies (H4) by Proposition \ref{xiH4}. We take $N = 1$ in Proposition \ref{circlemethod}, and we compute $p_0 = a_0 c_{0,0} = \frac{1}{2} c_{0,0} = \frac{\sqrt[4]{3}}{2 \pi \sqrt[4]{2}}$. In turn, we obtain the result in Theorem \ref{theorem2}(1). Similarly, $L_t(q)$ satisfies (H1) by Lemma \ref{LH1} with $a_0 = \beta^*_t$ and $a_k = 0$ for $k > 0$, and $L_t(q)$ satisfies (H3) by Lemma \ref{LH3}. We apply Proposition \ref{circlemethod} with $N = 1$ and produce the asymptotic in Theorem \ref{theorem2}(2).

\subsection{A Probabilistic Consequence}

It is natural to consider implications of Theorem \ref{theorem2}(1) and (2) to probabilistic features of the hook numbers of $\mathcal{SC}(n)$ and $\mathcal{DO}(n)$. Here, we prove a corollary that establishes asymptotics for the average number of hooks. Borrowing notation from \cite{CDH}, let $\textrm{avg}_{\mathcal{L}}(t; n)$ be the average of $n_t(\lambda)$ across partitions $\lambda$ with $|\lambda| = n$ in the collection $\mathcal{L}$. 

\begin{cor}\label{avgcor} We have the following asymptotics: 
\[ \textrm{avg}_{\mathcal{SC}}(t;n) \sim \frac{2}{\pi} \sqrt{\frac{3n}{2}}, \quad \textrm{avg}_{\mathcal{DO}}(t;n) \sim \frac{4 \beta_t^*}{\pi} \sqrt{\frac{3n}{2}}. \]

\end{cor}

\begin{proof}
Let $p_{\mathcal{SC}}(n) = |\mathcal{\mathcal{SC}}(n)|$ and $p_{\mathcal{DO}}(n) = |\mathcal{DO}(n)|$. It is a well-known fact that 
\begin{equation}\label{totalasymp} p_{\mathcal{SC}}(n) = p_{\mathcal{DO}}(n) \sim \frac{\sqrt[4]{2}}{4 \sqrt[4]{3} \cdot n^{\frac{3}{4}}} e^{ \pi \sqrt{n / 6 }}. \end{equation}

Since $(-q;q^2)_\infty$ is the generating function for $p_{\mathcal{SC}}(n) = p_{\mathcal{DO}}(n)$, this can be directly shown from \ref{xiH2} and \ref{xiH4} using Proposition \ref{circlemethod}. In light of Theorem \ref{theorem2}(1) and (2), \eqref{totalasymp} gives us the result.
\end{proof}

\section{Evaluating $\beta^*_t$}\label{sectionevaluatingbeta}

The object of this section is to evaluate the constant $\beta_t^*$, and our main result is the following formula. Since $\gamma_t^* = \frac{1}{2 \beta_t^*}$, Theorem \ref{betastart} additionally produces an explicit formula for $\gamma_t^*$. 

\begin{thm}\label{betastart} 
When $t$ is even, we have that
\small
\begin{align*} &\beta^*_t  = \frac{1}{2} \Bigg( \sum_{\substack{n \equiv \frac{t}{2} - 2 \pmod{3} \\ n > 0 }}^{\frac{t}{2}-2} \frac{1}{n} - \sum_{\substack{n \equiv \frac{t}{2} - 2 \pmod{3} \\ n > 0 }}^{\frac{t}{2} - 2} \frac{1}{n} \left( \frac{1}{2} + \mathbbm{1}_{2 \mid n} \frac{(-1)^{n/2}}{2^n} - \frac{3}{2} \mathbbm{1}_{3 \mid n} \frac{1}{2^{2n/3}} \right) + \mathbbm{1}_{6 \mid (t-4)} \int_{1/2}^1 \frac{(1-x)^{\frac{t+2}{3}-1}}{x} dx \\
&- \sum_{\substack{n \equiv \frac{t}{2}-2 \pmod{3} \\ n > 0 }}^{\frac{t}{2} - 2} \frac{1}{n+3/2} \left( \frac{1}{2} + \frac{1}{4} \mathbbm{1}_{2 \mid n} \frac{(-1)^{n/2}}{2^n} + \frac{1}{4} \mathbbm{1}_{2 \mid (n-1)} \frac{(-1)^{(n-1)/2}}{2^n} - \frac{3}{4} \mathbbm{1}_{3 \mid n} \frac{1}{2^{2n/3}} \right) \\
&+ \sum_{\substack{n \equiv \frac{t}{2} - 1 \pmod{3} \\ n > 0 }}^{\frac{t}{2}-1} \frac{1}{n} - \sum_{\substack{n \equiv \frac{t}{2} - 1 \pmod{3} \\ n > 0 }}^{\frac{t}{2} - 1} \frac{1}{n} \left( \frac{1}{2} + \mathbbm{1}_{2 \mid n} \frac{(-1)^{n/2}}{2^n} - \frac{3}{2} \mathbbm{1}_{3 \mid n} \frac{1}{2^{2n/3}} \right) + \mathbbm{1}_{6 \mid (t-2)}   \int_{1/2}^1 \frac{(1-x)^{\frac{t+1}{3}-1}}{x} dx \\
&- \sum_{\substack{n \equiv \frac{t}{2}-1 \pmod{3} \\ n > 0 }}^{\frac{t}{2} - 4} \frac{1}{n+3/2} \left( \frac{1}{2} + \frac{1}{4} \mathbbm{1}_{2 \mid n} \frac{(-1)^{n/2}}{2^n} + \frac{1}{4} \mathbbm{1}_{2 \mid (n-1)} \frac{(-1)^{(n-1)/2}}{2^n} - \frac{3}{4} \mathbbm{1}_{3 \mid n} \frac{1}{2^{2n/3}} \right) \Bigg). \end{align*}
\normalsize

When $t$ is odd, we have that
\small
\begin{align*} \beta^*_t  &= \frac{1}{2} \Bigg( \sum_{\substack{n \equiv \frac{t-1}{2} \pmod{3} \\ n > 0 }}^{\frac{t-1}{2}} \frac{1}{n} - \sum_{\substack{n \equiv \frac{t-1}{2} \pmod{3} \\ n > 0 }}^{\frac{t-1}{2}} \frac{1}{n} \left( \frac{1}{2} + \mathbbm{1}_{2 \mid n} \frac{(-1)^{n/2}}{2^n} - \frac{3}{2} \mathbbm{1}_{3 \mid n} \frac{1}{2^{2n/3}} \right) + \mathbbm{1}_{6 \mid (t-1)}   \int_{1/2}^1 \frac{(1-x)^{\frac{t+2}{3}-1}}{x} dx \\
&- \sum_{\substack{n \equiv \frac{t-1}{2} \pmod{3} \\ n > 0 }}^{\frac{t-1}{2}-3} \frac{1}{n+3/2} \left( \frac{1}{2} + \frac{1}{4} \mathbbm{1}_{2 \mid n} \frac{(-1)^{n/2}}{2^n} + \frac{1}{4} \mathbbm{1}_{2 \mid (n-1)} \frac{(-1)^{(n-1)/2}}{2^n} - \frac{3}{4} \mathbbm{1}_{3 \mid n} \frac{1}{2^{2n/3}} \right)  \\ &+ \sum_{\substack{n \equiv \frac{t-5}{2} \pmod{3} \\ n > 0 }}^{\frac{t-5}{2}} \frac{1}{n} - \sum_{\substack{n \equiv \frac{t-5}{2} \pmod{3} \\ n > 0 }}^{\frac{t-5}{2}} \frac{1}{n} \left( \frac{1}{2} + \mathbbm{1}_{2 \mid n} \frac{(-1)^{n/2}}{2^n} - \frac{3}{2} \mathbbm{1}_{3 \mid n} \frac{1}{2^{2n/3}} \right) + \mathbbm{1}_{6 \mid (t-5)} \int_{1/2}^1 \frac{(1-x)^{\frac{t+1}{3}-1}}{x} dx \\
&- \sum_{\substack{n \equiv \frac{t-5}{2} \pmod{3} \\ n > 0 }}^{\frac{t-5}{2}} \frac{1}{n+3/2} \left( \frac{1}{2} + \frac{1}{4} \mathbbm{1}_{2 \mid n} \frac{(-1)^{n/2}}{2^n} + \frac{1}{4} \mathbbm{1}_{2 \mid (n-1)} \frac{(-1)^{(n-1)/2}}{2^n} - \frac{3}{4} \mathbbm{1}_{3 \mid n} \frac{1}{2^{2n/3}} \right) \Bigg). \end{align*}
\end{thm}

\begin{cor}\label{cor} We have the following. 
\begin{enumerate} \item $\beta_t^* \in \mathbb{Q}(\log(2))$ for all $t$,
\item $\beta_t^* \in \mathbb{Q}$ if and only if $t$ is a multiple of $3$. 
\end{enumerate}
\end{cor}

\begin{proof}
    The only irrational terms in Theorem \ref{betastart} come from $\frac{1}{x}$ in the expansion of the integrand of 
    \[ \int_{1/2}^1 \frac{(1-x)^{\lceil \frac{t}{3} \rceil -1}}{x} dx \]
    and are equal to $\log(2)$. An integral of this form emerges if and only if $t$ belongs to a residue mod $6$ that is not divisible by $3$. 
\end{proof}

To compute $\beta_t^*$, we introduce the parameters $r$ and $s$ and evaluate the quantity 
\begin{equation*} S_{r,s}(t) := \sum_{k = 0}^{\lfloor (t-2r+s)/6 \rfloor} \binom{\frac{t-2k-2+s}{2}}{2k+r-1} I \left(2(2k+r), \frac{t-2k +s}{2} \right). \end{equation*}
From \eqref{betastarteven} and \eqref{betastartodd}, we have that 
\begin{equation}\label{betainS}
    \beta_t^* = \begin{cases} S_{2,0}(t) + S_{1,0}(t) \text{ for } t \text{ even,} \\ S_{1,1}(t) + S_{2,-1}(t) \text{ for } t \text{ odd.} \end{cases}
\end{equation}

\subsection{Simplifying $S_{r,s}(t)$}

\begin{prop}\label{Srst} Let $t \geq 2$ be a positive integer. Then, we have  
\begin{equation*} S_{r,s}(t) = \frac{1}{2} \sum_{k=0}^{\lfloor \frac{t-2r+s}{6} \rfloor} \binom{ \frac{t-2k -2 + s}{2}}{2k + r -1} \sum_{i=0}^{2k+r-1} \binom{2k+r-1}{i} \frac{(-1)^i}{\frac{t+s}{2} - 3k - r + i} \left( 1 - \frac{1}{2^{\frac{t+s}{2} -3k -r + i}} \right) \end{equation*}
when $t - 2r + s$ is not a multiple of $6$ and \begin{align*} S_{r,s}(t) = \frac{1}{2} \Bigg( \log 2 &+ \sum_{k=0}^{\frac{t-2r+s}{6} - 1} \binom{ \frac{t-2k -2 + s}{2}}{2k + r -1} \sum_{i=0}^{2k+r-1} \binom{2k+r-1}{i} \frac{(-1)^i}{\frac{t+s}{2} - 3k - r + i} \left( 1 - \frac{1}{2^{\frac{t+s}{2} -3k -r + i}} \right) 
\\ &+ \sum_{i=1}^{\frac{t+r+s}{3}-1} \binom{\frac{t+r+s}{3}-1}{i} \frac{(-1)^i}{i} \left( 1 - \frac{1}{2^i} \right) \Bigg). \end{align*} 
\end{prop}

\begin{proof} We begin by computing the integral $I \left(2(2k+r), \frac{t-2k +s}{2} \right)$. Substituting $u = e^{-2x} + 1$ yields
\begin{equation*} I \left(2(2k+r), \frac{t-2k +s}{2} \right) = \frac{1}{2} \int_1^2 \frac{(u-1)^{2k + r - 1}}{u^{\frac{t-2k+s}{2}}} du = \sum_{i=0}^{2k+r-1} \frac{(-1)^i}{2} \binom{2k+r-1}{i} \int_1^2 u^{3k + r - 1 - \frac{t+s}{2} - i} du. \end{equation*}

We focus on the inner integral. Since $k \leq \lfloor (t-2r+s)/6 \rfloor \leq (t-2r+s)/6$, we have that $3k + r - 1 - \frac{t+s}{2} - i \leq -i$ with equality only if $k = (t - 2r + s) /6$. In particular, we have that $3k + r - 1 - \frac{t+s}{2} - i = 0$ if and only if $(k,i) = \left((t - 2r + s)/6, 0 \right)$. Thus, we obtain
\begin{align*} 
\int_1^2 u^{3k + r - 1 - \frac{t+s}{2} - i} du = \begin{cases} 
\frac{1}{\frac{t+s}{2} - 3k - r + i} \left( 1 - \frac{1}{2^{\frac{t+s}{2} - 3k - r + i}} \right) & \text{ if } (k, i) \neq \left( \frac{t-2r+s}{6}, 0 \right), \\ \log 2 & \text{ if } (k, i) = \left( \frac{t-2r+s}{6}, 0 \right). 
\end{cases} 
\end{align*}
Thus, if $t - 2r + s$ is not a multiple of $6$, the $\log(2)$ term does not appear and we have
\begin{equation*} S_{r,s}(t) = \frac{1}{2} \sum_{k=0}^{\lfloor \frac{t-2r+s}{6} \rfloor} \binom{ \frac{t-2k -2 + s}{2}}{2k + r -1} \sum_{i=0}^{2k+r-1} \binom{2k+r-1}{i} \frac{(-1)^i}{\frac{t+s}{2} - 3k - r + i} \left( 1 - \frac{1}{2^{\frac{t+s}{2} -3k -r + i}} \right).
\end{equation*}
If $t - 2r + s$ is a multiple of $6$, we isolate the term associated with $\left( \frac{t-2r+s}{6}, 0 \right)$ and obtain 
\begin{align*} S_{r,s}(t) = \frac{1}{2} \Bigg( \log 2 &+ \sum_{k=0}^{\frac{t-2r+s}{6} - 1} \binom{ \frac{t-2k -2 + s}{2}}{2k + r -1} \sum_{i=0}^{2k+r-1} \binom{2k+r-1}{i} \frac{(-1)^i}{\frac{t+s}{2} - 3k - r + i} \left( 1 - \frac{1}{2^{\frac{t+s}{2} -3k -r + i}} \right) 
\\ &+ \sum_{i=1}^{\frac{t+r+s}{3}-1} \binom{\frac{t+r+s}{3}-1}{i} \frac{(-1)^i}{i} \left( 1 - \frac{1}{2^i} \right) \Bigg), \end{align*}
where the last term comes from substituting $k = \frac{t -2r+s}{6}$. \end{proof}

%Craig, Dawsey, and Han prove the following lemma. 
%\begin{lem}\label{binomidentity1}
%\begin{equation*}
    %\sum_{k=0}^n \binom{n}{k} \frac{(-1)^k x^{m+k}}{m+k} = \frac{x^m  \sum_{j=0}^{n+m-1} (1-x)^{n-j} \binom{n+m-1-j}{m-1}}{m \binom{n+m}{m}}
%\end{equation*}
%\end{lem}

\subsection{A Combinatorial Interlude}

The following lemma will be useful in our computation of $\beta^*_t$. 

\begin{lem}\label{binomidentity}
Let $n > 0$ and set $n = 3p + q$. We have the two identities 
\begin{equation}\label{eveneq} \sum_{i=0}^{p+ \lfloor (q - 1)/3 \rfloor} \frac{1}{n-3i} \binom{n-i-1}{2i} 2^i = \frac{1}{n} \left( 2^{n-1} + \mathbbm{1}_{2 \mid n} (-1)^{n/2} + \mathbbm{1}_{3 \mid n} (-3 \cdot 2^{n/3 - 1} ) \right),  \end{equation} 
\begin{align}\label{oddeq} \sum_{i=0}^{p+ \lfloor (q - 1)/3 \rfloor} \frac{1}{n-3i} \binom{n-i}{2i+1} 2^i &= \frac{1}{n+3/2} \Bigg( 2^{n} + \mathbbm{1}_{2 \mid n} \left(\frac{1}{2} \cdot (-1)^{n/2} \right) \\ 
&+ \mathbbm{1}_{2 \mid (n-1)} \left( \frac{1}{2} \cdot (-1)^{n-1/2} \right) + \mathbbm{1}_{3 \mid n} (-3 \cdot 2^{n/3 - 1} ) \Bigg). \nonumber \end{align}
\end{lem} 

\begin{proof}
We consider the generating function for the quantity $\sum_{i=0}^n \binom{n - i-1}{2i} 2^i$, which can be written
\begin{equation*} \sum_{i \geq 0} (2x^3)^i \ \cdot \ x \sum_{n \geq 3i +1} \binom{n -i - 1}{2i} x^{n - 3i - 1} = \sum_{i \geq 0} (2x^3)^i \frac{x}{(1-x)^{2i+1}} = \frac{x}{5} \left( \frac{x+3}{1+x^2} + \frac{2}{1-2x} \right).
\end{equation*}
Using geometric series expansions, we obtain 
\begin{equation}\label{evenbin}
\sum_{i=0}^n \binom{n - i-1}{2i} 2^i = \begin{cases} \frac{1}{5} \left( 2^{2k} + (-1)^{k-1} \right)   \text{ if } n = 2k, \\ \frac{1}{5} \left( 2^{2k+1} + 3 (-1)^{k} \right) \text{ if } n = 2k + 1.
\end{cases} 
\end{equation}
Similarly, we compute 
\begin{equation}\label{oddbin}
\sum_{i=0}^n \binom{n - i}{2i+1} 2^i = \begin{cases}  \frac{1}{5} \left( 2^{2k+1} + 2 (-1)^{k-1} \right) \text{ if } n = 2k, \\  \frac{1}{5} \left( 2^{2k+2} + (-1)^{k} \right)  \text{ if } n = 2k + 1.
    \end{cases} 
\end{equation}
Note that $2i \leq n - i - 1$ and $2i+1 \leq n -i$ both imply $i \leq \frac{n-1}{3}$. Thus, we have the two expressions 
\begin{equation}\label{goaleq} \sum_{i=0}^{p+ \lfloor (q - 1)/3 \rfloor} \binom{n - i-1}{2i} 2^i = \sum_{i=0}^n \binom{n - i-1}{2i} 2^i \quad \text{and } \sum_{i=0}^{p+ \lfloor (q - 1)/3 \rfloor} \binom{n - i}{2i+1} 2^i = \sum_{i=0}^n \binom{n - i}{2i+1} 2^i \end{equation} 

We now focus on \eqref{eveneq}. To introduce the fraction $\frac{1}{n-3i}$, we write it as $\frac{1}{n} + \frac{3i}{n(n-3i)}$, and we find 
\begin{equation*} \sum_{i=0}^{p+ \lfloor (q - 1)/3 \rfloor} \frac{1}{n-3i} \binom{n - i - 1}{2i} 2^i = \frac{1}{n} \sum_{i=0}^{p+ \lfloor (q - 1)/3 \rfloor} \binom{n-i-1}{2i} 2^i + \frac{3}{2n} \sum_{i=0}^{p+ \lfloor (q - 1)/3 \rfloor} \binom{n-i-1}{2i-1} 2^i \end{equation*}
\begin{equation*} = \frac{1}{n} \sum_{i=0}^{p+ \lfloor (q - 1)/3 \rfloor} \binom{n-i-1}{2i} 2^i + \frac{3}{n} \sum_{i=0}^{p+ \lfloor (q - 1)/3 \rfloor - 1} \binom{(n-2)-i}{2i+1} 2^i. \end{equation*}
The expression $2i + 1 \leq n - i -2$ implies that $i \leq \frac{n-3}{3}$. Thus, the sum $\sum_{i=0}^{{p+ \lfloor (q - 1)/3 \rfloor}-1} \binom{n - 2- i}{2i+1} 2^{i}$ is equal to $\sum_{i=0}^n \binom{n-2-i}{2i+1} 2^i$ if $n \equiv 1, 2 \pmod{3}$, as then $i \leq \frac{n-3}{3}$ with $i \in \ZZ$ implies that $i \leq p+ \lfloor (q - 1)/3 \rfloor - 1$. Otherwise, if $n \equiv 0$, we have that $i = (n-3)/3$ is included in the sum $\sum_{i=0}^n \binom{n-2 - i}{2i+1} 2^{i}$ and not the sum $\sum_{i=0}^{{p+ \lfloor (q - 1)/3 \rfloor}-1} \binom{n-2-i}{2i+1} 2^{i}$. We obtain 
\begin{equation}\label{adjust} \sum_{i=0}^{{p+ \lfloor (q - 1)/3 \rfloor}-1} \binom{n-2-i}{2i+1} 2^{i} = \begin{cases}  \sum_{i=0}^{n} \binom{n-2 -i}{2i+1} 2^{i} - 2^{(n-3)/3} &\text{ if } n \equiv 0 \pmod{3}, \\ \sum_{i=0}^{n} \binom{n-2 -i}{2i+1} 2^{i} &\text{ if } n \equiv 1, 2 \pmod{3}. \end{cases} \end{equation}
Substituting \eqref{evenbin}, \eqref{oddbin}, and \eqref{adjust} into \eqref{goaleq} yields the following: 
\begin{align*} \sum_{i=0}^{p+ \lfloor (q - 1)/3 \rfloor} \frac{1}{n-3i} \binom{n - i - 1}{2i} 2^i = \begin{cases}  \frac{1}{n} \left( 2^{n-1} +  (-1)^k - 3 \cdot 2^{2k-1} \right) &\text{ if } n = 6k, \\ \frac{1}{n} \left( 2^{n-1} \right) &\text{ if } n = 6k+1, \\  \frac{1}{n} \left( 2^{n-1} + (-1)^{k+1} \right) &\text{ if } n = 6k + 2,
\\  \frac{1}{n} \left( 2^{n-1} - 3 \cdot 2^{2k} \right) &\text{ if } n = 6k + 3, \\ \frac{1}{n} \left( 2^{n-1} +  (-1)^{k+2} \right) &\text{ if } n = 6k +4, \\  \frac{1}{n} \left( 2^{n-1} \right) &\text{ if } n = 6k + 5. \end{cases} \end{align*}
This proves \eqref{eveneq}. Similarly, the expression $\sum_{i=0}^{p+ \lfloor (q - 1)/3 \rfloor} \binom{n - i}{2i+1} 2^i$ depends on both the parity of $n$ and its residue mod $3$, and we obtain \eqref{oddeq}.
\end{proof}

\subsection{Proof of Theorem \ref{betastart}}

\begin{prop}\label{almostthetheorem} Let $x = \frac{t+s}{2} - r$ with $t \equiv s \pmod{2}$, and let the quantity $Q_{r,s}(t)$ be
\begin{align*} Q_{r,s}(t) &:= \frac{1}{2} \sum_{\substack{n \equiv x \pmod{3} \\ n > 0 }}^x \frac{1}{n} - \frac{1}{2} \sum_{\substack{n \equiv x \pmod{3} \\ n > 0 }}^{x + 3 \lfloor \frac{r-1}{2} \rfloor} \frac{1}{n} \left( \frac{1}{2} + \mathbbm{1}_{2 \mid n} \frac{(-1)^{n/2}}{2^n} - \frac{3}{2} \mathbbm{1}_{3 \mid n} \frac{1}{2^{2n/3}} \right) \\
&- \frac{1}{2} \sum_{\substack{n \equiv x \pmod{3} \\ n > 0 }}^{x + 3 \lfloor \frac{r-2}{2} \rfloor} \frac{1}{n+3/2} \left( \frac{1}{2} + \frac{1}{4} \mathbbm{1}_{2 \mid n} \frac{(-1)^{n/2}}{2^n} + \frac{1}{4} \mathbbm{1}_{2 \mid (n-1)} \frac{(-1)^{(n-1)/2}}{2^n} - \frac{3}{4} \mathbbm{1}_{3 \mid n} \frac{1}{2^{2n/3}} \right). \end{align*} 
We prove that $S_{r,s}(t)$ to $Q_{r,s}(t)$ are related as follows:
\begin{enumerate}
\item If $t - 2r + s$ is not a multiple of $6$, $S_{r,s}(t) = Q_{r,s}(t)$. 
\item If $t - 2r + s$ is a multiple of $6$, 
\begin{equation*} S_{r,s}(t) = Q_{r,s}(t) + \frac{1}{2} \int_{1/2}^1 \frac{(1-x)^{\frac{t+r+s}{3}-1}}{x} dx. \end{equation*}
\end{enumerate}
\end{prop}

\begin{proof} We begin with the case where $t - 2r + s$ is not a multiple of $6$. By Proposition \ref{Srst}, we have
\begin{equation*} S_{r,s}(t) = \frac{1}{2} \sum_{k=0}^{\lfloor \frac{t-2r+s}{6} \rfloor} \binom{ \frac{t-2k -2 + s}{2}}{2k + r -1} \sum_{i=0}^{2k+r-1} \binom{2k+r-1}{i} \frac{(-1)^i}{\frac{t+s}{2} - 3k - r + i} \left( 1 - \frac{1}{2^{\frac{t+s}{2} -3k -r + i}} \right). \end{equation*}
We can simplify the sum that constitutes a portion of this expression, using Lemma 4.4.1 in \cite{CDH}:
\begin{equation*} S^{(1)}_{r,s}(t) := \frac{1}{2} \sum_{k=0}^{\lfloor \frac{t-2r+s}{6} \rfloor} \binom{ \frac{t-2k -2 + s}{2}}{2k + r -1} \sum_{i=0}^{2k+r-1} \binom{2k+r-1}{i} \frac{(-1)^i}{\frac{t+s}{2} - 3k - r + i} = \frac{1}{2} \sum_{k=0}^{\lfloor \frac{t-2r+s}{6} \rfloor} \frac{1}{\frac{t+s}{2} - 3k - r }. \end{equation*}
Now, we consider the other portion of the sum $S^{(2)}_{r,s}(t)$ defined so that $S_{r,s}(t) = S^{(1)}_{r,s}(t) - S^{(2)}_{r,s}(t)$:
\begin{equation*} S^{(2)}_{r,s}(t) := \frac{1}{2} \sum_{k=0}^{\lfloor \frac{t-2r+s}{6} \rfloor} \binom{ \frac{t-2k -2 + s}{2}}{2k + r -1} \sum_{i=0}^{2k+r-1} \binom{2k+r-1}{i} \frac{(-1)^i}{\frac{t+s}{2} - 3k - r + i} \left(\frac{1}{2^{\frac{t+s}{2} -3k -r + i}} \right). \end{equation*}
Applying Lemma 4.4.1 in \cite{CDH} on the inner sum and reindexing, we obtain 
\begin{align*} &\sum_{i=0}^{2k+r-1} \binom{2k+r-1}{i} \frac{(-1)^i}{\frac{t+s}{2} - 3k - r + i} \left(\frac{1}{2^{\frac{t+s}{2} -3k -r + i}} \right) \\
&= \frac{1}{2} \sum_{k=0}^{\lfloor \frac{t-2r+s}{6} \rfloor} \frac{1}{\frac{t+s}{2} - 3k - r} \sum_{i=0}^{2k+r-1} \frac{1}{2^{\frac{t+s}{2} - k - 1 - i}} \binom{\frac{t+s}{2} - k - 2 - i}{\frac{t+s}{2} - 3k - r - 1} \\ &= \frac{1}{2} \sum_{k=0}^{\lfloor \frac{t-2r+s}{6} \rfloor} \frac{1}{\frac{t+s}{2} - 3k - r} \sum_{i=0}^{2k+r-1} \frac{1}{2^{\frac{t+s}{2} - 3k - r + i}} \binom{\frac{t+s}{2} - 3k - r - 1 +i}{i} \\ &= \frac{1}{2} \sum_{i=0}^{2 \lfloor \frac{t-2r+s}{6} \rfloor + r -1} \sum_{k= \lceil \frac{i - r + 1}{2} \rceil}^{\lfloor \frac{t-2r+s}{6} \rfloor} \frac{1}{\frac{t+s}{2} - 3k - r}  \binom{\frac{t+s}{2} - 3k - r - 1 +i}{i} \frac{1}{2^{\frac{t+s}{2} - 3k - r + i}}, \end{align*}
switching the order of summation in the last line. We begin transforming this expression into the form of Lemma \ref{binomidentity}, letting $3a + q = \frac{t+s}{2} - r$ where $q \in \{0, 1, 2 \}$, yielding
\begin{equation*} \frac{1}{2} \sum_{i=0}^{2a+r-1} \sum_{k = \lceil \frac{i-r+1}{2} \rceil}^a \frac{1}{3a + q - 3k} \binom{3a-3k + q  - 1 +i}{i} \frac{1}{2^{3a-3k+q+i}}. \end{equation*}
Further, we substitute $b = a -k$, and we get 
\begin{equation*} \frac{1}{2} \sum_{i=0}^{2a+r-1} \sum_{b = 0}^{a - \lceil \frac{i-r+1}{2} \rceil} \frac{1}{3b + q} \binom{3b + q - 1 +i}{i} \frac{1}{2^{3b+q + i}}. \end{equation*}
Splitting into cases when $i = 2i'$ is even and $i = 2i' + 1$ is odd, we rewrite the sum as follows: 
\begin{align*} \frac{1}{2} \sum_{i'=0}^{a + \lfloor \frac{r-1}{2} \rfloor} \sum_{b = 0}^{a - i' + \lfloor \frac{r-1}{2} \rfloor} \frac{1}{3b + q} \binom{3b + q - 1 + 2i'}{2i'} \frac{1}{2^{3b+q + 2i'}} \\ + \ \frac{1}{2} \sum_{i'=0}^{a+ \lfloor \frac{r-2}{2} \rfloor} \sum_{b = 0}^{a - i' + \lfloor \frac{r-2}{2} \rfloor} \frac{1}{3b + q} \binom{3b + q + 2i'}{2i' + 1} \frac{1}{2^{3b+q + 2i' + 1}}. \end{align*}
Finally, we consider the quantity $p = b + i'$ and we reindex the sum by $p$ and $i'$, giving
\begin{align*} \frac{1}{2} \sum_{p=0}^{a + \lfloor \frac{r-1}{2} \rfloor}  \sum_{i'= 0}^{p + \lfloor (q-1)/3 \rfloor} \frac{1}{3p - 3i' + c} \binom{3p - i' + q - 1}{2i'} \frac{1}{2^{3p-i'+q}} \\ + \ \frac{1}{2} \sum_{p=0}^{a + \lfloor \frac{r-2}{2} \rfloor} \sum_{i'= 0}^{p + \lfloor (q-1)/3 \rfloor } \frac{1}{3p - 3i' + q} \binom{3p - i' + q}{2i'+1} \frac{1}{2^{3p-i'+q + 1}}. \end{align*}
Reindexing the sums gives $i' \leq \min \left( \lfloor a/2 \rfloor, p \right)$. The binomial coefficients are nonzero when $i' \leq p + \lfloor q-1/3 \rfloor$, which is less than $\min \left( \lfloor a/2 \rfloor, p \right)$, justifying the limits on the sums. We now apply Lemma \ref{binomidentity} to write the first summand as
\begin{align*} \sum_{\substack{n \equiv q (\text{mod} 3) \\ n > 0 }}^{3a + 3 \lfloor \frac{r-1}{2} \rfloor + q} \frac{1}{n} \left( \frac{1}{2} + \mathbbm{1}_{2 \mid n} \frac{(-1)^{n/2}}{2^n} - \frac{3}{2} \mathbbm{1}_{3 \mid n} \frac{1}{2^{2n/3}} \right). \end{align*}
Similarly, we write the second summand as 
\begin{align*} \sum_{\substack{n \equiv q (\text{mod} 3) \\ n > 0 }}^{3a + 3 \lfloor \frac{r-2}{2} \rfloor + q} \frac{1}{n+3/2} \left( \frac{1}{2} + \frac{1}{4} \mathbbm{1}_{2 \mid n} \frac{(-1)^{n/2}}{2^n} + \frac{1}{4} \mathbbm{1}_{2 \mid (n-1)} \frac{(-1)^{(n-1)/2}}{2^n} - \frac{3}{4} \mathbbm{1}_{3 \mid n} \frac{1}{2^{2n/3}} \right). \end{align*}
Substituting $x = 3a + q = \frac{t+s}{2} - r$ and combining with the original first term gives the result.

In the case when $t-2r+s$ is a multiple of $6$, we have $q = 0$. From Proposition \ref{Srst}, we have
\begin{align*} S_{r,s}(t) = \frac{1}{2} \Bigg( \log(2) &+ \sum_{k=0}^{\frac{t-2r+s}{6} - 1} \binom{ \frac{t-2k -2 + s}{2}}{2k + r -1} \sum_{i=0}^{2k+r-1} \binom{2k+r-1}{i} \frac{(-1)^i}{\frac{t+s}{2} - 3k - r + i} \left( 1 - \frac{1}{2^{\frac{t+s}{2} -3k -r + i}} \right) 
\\ &+ \sum_{i=1}^{\frac{t+r+s}{3}-1} \binom{\frac{t+r+s}{3}-1}{i} \frac{(-1)^i}{i} \left( 1 - \frac{1}{2^i} \right) \Bigg). \end{align*}
We isolate the term $(i,k) = (0, \frac{t+r+s}{3})$, set $q = 0$, and reindex the above computations to find
\begin{align}\label{mult6S}
     S_{r,s}(t) = \frac{1}{2} \Bigg( \log(2) &+ \sum_{\substack{n \equiv 0 \pmod{3} \\ n > 0 }}^{\frac{t+s}{2} -r} \frac{1}{n} - \sum_{\substack{n \equiv 0 (\text{mod} 3) \\ n > 0 }}^{\frac{t+s}{2} -r + 3 \lfloor \frac{r-1}{2} \rfloor} \frac{1}{n} \left( \frac{1}{2} + \mathbbm{1}_{2 \mid n} \frac{(-1)^{n/2}}{2^n} - \frac{3}{2} \mathbbm{1}_{3 \mid n} \frac{1}{2^{2n/3}} \right) \nonumber \\ &- \sum_{\substack{n \equiv 0 (\text{mod} 3) \\ n > 0 }}^{\frac{t+s}{2} -r + 3 \lfloor \frac{r-2}{2} \rfloor} \frac{1}{n+3/2} \left( \frac{1}{2} + \frac{1}{4} \mathbbm{1}_{2 \mid n} \frac{(-1)^{n/2}}{2^n} + \frac{1}{4} \mathbbm{1}_{2 \mid (n-1)} \frac{(-1)^{(n-1)/2}}{2^n} - \frac{3}{4} \mathbbm{1}_{3 \mid n} \frac{1}{2^{2n/3}} \right) \\ \nonumber
     &+ \sum_{i=1}^{\frac{t+r+s}{3}-1} \binom{\frac{t+r+s}{3}-1}{i} \frac{(-1)^i}{i} \left( 1 - \frac{1}{2^i} \right) \Bigg). \nonumber
\end{align}
Using Lemma 4.4(2) in \cite{CDH}, the sum of the first and last terms becomes 
\begin{equation*} \log 2 + \sum_{i=1}^{\frac{t+r+s}{3}-1} \binom{\frac{t+r+s}{3}-1}{i} \frac{(-1)^i}{i} \left( 1 - \frac{1}{2^i} \right) = \log 2 + \int_{1/2}^1 \frac{(1-x)^{\frac{t+r+s}{3}-1}-1}{x} dx =  \int_{1/2}^1 \frac{(1-x)^{\frac{t+r+s}{3}-1}}{x} dx. \end{equation*}
Substituting this into \eqref{mult6S} gives the result. \end{proof}

Together with \eqref{betainS}, Proposition \ref{almostthetheorem} demonstrates Theorem \ref{betastart}.

\section{Proof of Theorem \ref{theorem2}(3)}\label{sectionconj}

We demonstrate that $\beta^*_t < \frac{1}{2}$ for all $t$. To do so, we consider the fluctuations of $\beta_t^*$ for large $t$. Define the auxiliary functions $f_1(n), f_2(n), f_3(n)$ as 
\begin{equation*}
    f_1(n) = \frac{1}{2n} - \mathbbm{1}_{2 \mid n} \frac{(-1)^{n/2}}{n 2^n} + \mathbbm{1}_{3 \mid n} \frac{3}{2n \cdot 2^{2n/3}},
\end{equation*} 
\begin{equation*} f_2(n) = - \frac{1}{2(n+3/2)} - \mathbbm{1}_{2 \mid n} \frac{(-1)^{n/2}}{4(n+3/2) 2^n} + \mathbbm{1}_{2 \mid n} \frac{(-1)^{(n-1)/2}}{4(n+3/2) 2^n} + \mathbbm{1}_{3 \mid n} \frac{3}{4(n+3/2) \cdot 2^{2n/3}}, \end{equation*}
\begin{equation*}
f_3(n) = \begin{cases} \displaystyle{\int_{1/2}^1 \frac{(1-x)^{\lfloor n/3 \rfloor}}{x} dx}\ \text{if} \ n \equiv 1, 2 \pmod{3}, \\ 0 \ \text{if} \ n \equiv 0 \pmod{3}. \end{cases}
\end{equation*}
For even $t$, we have the recurrence
\begin{equation*} \beta_t^* = \beta_{t-6}^* + f_1 \left(\frac{t}{2} - 2 \right) + f_1 \left(\frac{t}{2} - 1\right) + f_2\left(\frac{t}{2} - 2\right) + f_2\left(\frac{t}{2} - 4\right) + f_3\left(t\right). \end{equation*}
For odd $t$, we have, similiarly, that
\begin{equation*} \beta_t^* = \beta_{t-6}^* + f_1\left(\frac{t-1}{2}\right) + f_1\left(\frac{t-5}{2}\right) + f_2\left(\frac{t-1}{2} - 3\right) + f_2\left(\frac{t-5}{2}\right) + f_3\left(t\right). \end{equation*}
Suppose that $n \geq 3/2$. Accounting for pairs of terms in the recurrences, we bound the sums, using 
\begin{align}\label{bound1} f_1\left(n\right) + f_2\left(n\right) &\leq \frac{3}{4n(n+3/2)} + \frac{1}{n 2^n} + \frac{3}{2 n \cdot 2^{2n/3}} + \frac{1}{4(n+3/2)2^n} + \frac{3}{4(n+3/2) \cdot 2^{2n/3}} \nonumber \\  &\leq \frac{3}{4n^2} + \frac{5}{4 n 2^n} + \frac{9}{4 n \cdot 2^{2n/3}}, \end{align}
\begin{align}\label{bound2} f_1(n) + f_2(n-3) &\leq \frac{-3}{4n(n-3/2)} + \frac{1}{n 2^n} + \frac{3}{2 n \cdot 2^{2n/3}} + \frac{2}{(n-3/2)2^n} + \frac{1}{4(n-3/2) \cdot 2^{2n/3}} \nonumber \\
&\leq \frac{5}{4 (n-3/2) 2^n} + \frac{1}{4 (n-3/2) \cdot 2^{2n/3}}. \end{align}
Moreover, since $(1-x)^y/x$ is a nonnegative decreasing function of $y$, we have that
\begin{equation}\label{bound3} f_3(t) \leq \int_{1/2}^1 \frac{(1-x)^{\frac{t}{3}-1}}{x} dx \leq \int_{1/2}^1 \frac{(1/2)^{\frac{t}{3}}-1}{x} dx = \frac{2 \ln 2}{2^{\frac{t}{3}}}.  \end{equation}

Fix odd $t' \in \mathbb{N}$. For each odd $t \geq t' \pmod{6}$, we crudely bound $\beta^*_t$, including all the residues modulo 3 in the sum, using \eqref{bound1}, \eqref{bound2}, and \eqref{bound3}, and find
\small
\begin{align} \beta^*_t &= \beta^*_{t'} + \sum_{\substack{t \equiv t' \pmod{6} \\ t > t'}} f_1(\frac{t-1}{2}) + f_1\left(\frac{t-5}{2}\right) + f_2\left(\frac{t-1}{2} - 3\right) + f_2\left(\frac{t-5}{2}\right) + f_3\left(t\right) \\ &\leq \beta^*_{t'} +
\sum_{\substack{t \equiv t' \pmod{2} \\ t > t'}} f_1\left(\frac{t-1}{2}\right) + f_1\left(\frac{t-5}{2}\right) + f_2\left(\frac{t-1}{2} - 3\right) + f_2\left(\frac{t-5}{2}\right) + f_3\left(t\right) \nonumber \\
&\leq \beta^*_{t'} + \sum_{n \geq \frac{t'-5}{2}} \frac{3}{4n^2} + \frac{5}{4 n 2^n} + \frac{9}{4 n \cdot 2^{2n/3}} + \sum_{n \geq \frac{t'-1}{2}} \frac{5}{4 (n-3/2) 2^n} + \frac{1}{4 (n-3/2) \cdot 2^{2n/3}} + \sum_{n \geq t'} \frac{2 \ln 2}{2^{\frac{t}{3}}}. \nonumber
\end{align}

\normalsize
We now choose a suitable value of $t'$. Let $t' = 21$. Evaluating the above series yields that $\beta^*_t \leq \beta^*_{21} + 0.17052684...$ for $t > t'$, where $\beta^*_{21} = 0.30472711...$ Thus, for all odd $t > 21$, we have that $\beta^*_t \leq 0.47525396... < \frac{1}{2}$. A finite computational check demonstrates that $\beta^*_t < \frac{1}{2}$ for odd $t \leq 21$. 

Similarly, in the even case, we again fix $t' \in \NN$ and obtain that 
\begin{align*} \beta^*_t \leq \beta^*_{t'} + \sum_{n \geq \frac{t'}{2} - 2} \frac{3}{4n^2} + \frac{5}{4 n 2^n} + \frac{9}{4 n \cdot 2^{2n/3}} + \sum_{n \geq \frac{t'}{2} - 1} \frac{5}{4 (n-3/2) 2^n} + \frac{1}{4 (n-3/2) \cdot 2^{2n/3}} + \sum_{n \geq t'} \frac{2 \ln 2}{2^{\frac{t}{3}}}.
\end{align*} 
In this case, we take $t' = 20$. Computing the above series yields that $\beta^*_t \leq \beta^*_{20} + 0.18501868...$ for $t > t'$, where $\beta^*_{20} = 0.30607337...$ Thus, for all even $t > 20$, we have that $\beta^*_t \leq 0.49109205... < \frac{1}{2}$. As before, a finite computational check demonstrates that $\beta^*_t < \frac{1}{2}$ for even $t \leq 20$. 

\section{Proof of Theorem \ref{theorem3}}\label{sectionlimits}

To find a limit for $\gamma_t^*$ as $t \rightarrow \infty$, it suffices to find a limit for $2 \beta_t^*$ as $t \rightarrow \infty$. From Theorem \ref{betastart}, we can find limits for $2 \beta_t^*$ along $t \equiv t' \pmod{6}$ for $t' \in \mathbb{N}$.  As $t \rightarrow \infty$ along $t \equiv t' \pmod{6}$ with $t'$ even, we have that 
\small
\begin{align}\label{evenlimit} \lim_{\substack{t \rightarrow \infty \\ t \equiv t' \pmod{6}}} &2\beta^*_t  = \sum_{\substack{n \equiv \frac{t'}{2} - 2 \pmod{3} \\ n > 0 }}^{\infty} \frac{1}{n} - \sum_{\substack{n \equiv \frac{t'}{2} - 2 \pmod{3} \\ n > 0 }}^{\infty} \frac{1}{n} \left( \frac{1}{2} + \mathbbm{1}_{2 \mid n} \frac{(-1)^{n/2}}{2^n} - \frac{3}{2} \mathbbm{1}_{3 \mid n} \frac{1}{2^{2n/3}} \right) \nonumber \\
&- \sum_{\substack{n \equiv \frac{t'}{2}-2 \pmod{3} \\ n > 0 }}^{\infty} \frac{1}{n+3/2} \left( \frac{1}{2} + \frac{1}{4} \mathbbm{1}_{2 \mid n} \frac{(-1)^{n/2}}{2^n} + \frac{1}{4} \mathbbm{1}_{2 \mid (n-1)} \frac{(-1)^{(n-1)/2}}{2^n} - \frac{3}{4} \mathbbm{1}_{3 \mid n} \frac{1}{2^{2n/3}} \right) \\
&+ \sum_{\substack{n \equiv \frac{t'}{2} - 1 \pmod{3} \\ n > 0 }}^{\infty} \frac{1}{n} - \sum_{\substack{n \equiv \frac{t'}{2} - 1 \pmod{3} \\ n > 0 }}^{\infty} \frac{1}{n} \left( \frac{1}{2} + \mathbbm{1}_{2 \mid n} \frac{(-1)^{n/2}}{2^n} - \frac{3}{2} \mathbbm{1}_{3 \mid n} \frac{1}{2^{2n/3}} \right)\nonumber  \\
&- \sum_{\substack{n \equiv \frac{t'}{2}-1 \pmod{3} \\ n > 0 }}^{\infty} \frac{1}{n+3/2} \left( \frac{1}{2} + \frac{1}{4} \mathbbm{1}_{2 \mid n} \frac{(-1)^{n/2}}{2^n} + \frac{1}{4} \mathbbm{1}_{2 \mid (n-1)} \frac{(-1)^{(n-1)/2}}{2^n} - \frac{3}{4} \mathbbm{1}_{3 \mid n} \frac{1}{2^{2n/3}} \right) \nonumber. \end{align}
\normalsize
Here, we use the fact that the integral terms converge to $0$ as $t \rightarrow \infty$.
As $t \rightarrow \infty$ along $t \equiv t' \pmod{6}$ with $t'$ odd, a similar expression is derived from the second expression in Theorem 5.1. Define the auxiliary functions $g_1(k), g_2(k), g_3(k), g_4(k), g_5(k), g_6(k)$ as follows:
\begin{align*}
g_1(k) &= \sum_{\substack{n \equiv k \pmod{3} \\ n > 0}} \left( \frac{1}{n} - \frac{1}{2n} - \frac{1}{2(n+3/2)} \right) = \sum_{n \equiv k \pmod{3}} \frac{3}{2n(2n+3)}, \\
% subtracting 1/n - 1/n - 1/(n+3/2) sums
g_2(k) &= \sum_{\substack{n \equiv k \pmod{3} \\ n > 0}} \mathbbm{1}_{2 \mid n} \frac{(-1)^{n/2}}{n \cdot 2^n} = \sum_{m \equiv k/2 \pmod{3}} \frac{(-1)^m}{2m \cdot 2^{2m}},   \\
% plugging in n = 2n to first one!
g_3(k) &= \begin{cases} 
\displaystyle{\sum_{\substack{n \equiv 0 \pmod{3} \\ n > 0}} \frac{3}{2} \mathbbm{1}_{3 \mid n} \frac{1}{n \cdot 2^{2n/3}} = \sum_{m > 0} \frac{1}{2m \cdot 2^{2m}}} & \text{if } k \equiv 0 \pmod{3},  \\ 
0 & \text{otherwise,} 
\end{cases} \\
% subbing 3m for n in second one
g_4(k) &= \sum_{\substack{n \equiv k \pmod{3} \\ n > 0}} \frac{1}{4(n+3/2)} \mathbbm{1}_{2 \mid n} \frac{(-1)^{n/2}}{2^n} = \sum_{m \equiv k/2 \pmod{3}} \frac{1}{2(4m+3)} \frac{(-1)^m}{2^{2m}}, \\  
% subbing n = 2m to second one! 
g_5(k) &= \sum_{\substack{n \equiv k \pmod{3} \\ n > 0}} \frac{1}{4(n+3/2)} \mathbbm{1}_{2 \mid (n-1)} \frac{(-1)^{(n-1)/2}}{2^n} = \sum_{m \equiv (k-1)/2 \pmod{3}} \frac{1}{2(4m+5)} \frac{(-1)^m}{2^{2m+1}}, \\ 
% subbing 2m + 1
g_6(k) &= \begin{cases} 
\sum_{\substack{n \equiv k \pmod{3} \\ n > 0}} \frac{3}{4} \mathbbm{1}_{3 \mid n} \frac{1}{(n+3/2) \cdot 2^{2n/3}} = \sum_{m > 0} \frac{1}{2(2m+1) \cdot 2^{2m}} & \text{if } k \equiv 0 \pmod{3}, \\ 
0 & \text{otherwise.} 
\end{cases}
\end{align*}

Now, define a combination of the auxiliary functions by 
\begin{align*}
    G(k) := g_1(k) - g_2(k) + g_3(k) - g_4(k) - g_5(k) + g_6(k).
\end{align*}
We can now express the limit as
\begin{equation*}\label{Flimitexpression} \lim_{\substack{t \rightarrow \infty \\ t \equiv t' \pmod{6}}} 2\beta^*_t = \begin{cases} 
G\left(\frac{t}{2} - 2\right) + G\left(\frac{t}{2} - 1 \right) \textrm{ if } t \textrm{ is even}, \\
G\left(\frac{t-1}{2}\right) + G\left(\frac{t-5}{2}\right) \textrm{ if } t \textrm{ is odd}.
\end{cases} \end{equation*}
Note that each $g_i(k)$ is determined solely by the value of $k \pmod{3}$. Thus, it suffices to compute $g_i(k)$ for $i \in \{ 0,1,2, \dots, 6 \}$ and $k \in \{ 0, 1, 2 \}$. Below, for each fixed $i$, we compute $g_i(k)$ for all $k$.

\begin{enumerate}
\item[\textbf{(i)}] $g_1(k)$: Given the digamma function $\phi(s)$, we use the identity 
\begin{equation*} \phi(s) = -\gamma + \sum_{n=0}^\infty \left( \frac{1}{n+1} - \frac{1}{n+s} \right) \end{equation*}
where $\gamma$ is the Euler--Mascheroni constant. We find that
\small
\begin{equation*} g_1(0) = \sum_{n \equiv 0 \pmod{3}} \frac{3}{2n(2n+3)} = \sum_{m > 0} \frac{1}{2m(6m+3)}  = \frac{1}{6} \sum_{m \geq 0} \left( \frac{1}{m+1} - \frac{1}{m+3/2} \right) = \frac{1}{6} \left( \phi(3/2) + \gamma \right) = \frac{1}{3}(1 - \ln(2)) \end{equation*}
\normalsize
and similarly find that $g_1(1) = \displaystyle{\frac{1}{9} \left( \pi {\sqrt{3}} - 3 \ln(2) \right)}$ and $g_1(2) = \displaystyle{\frac{1}{9} \left(9 - \pi {\sqrt{3}} - 3 \ln(2) \right)}$.

\normalsize 

\item[\textbf{(ii)}] $g_2(k)$: When $k = 0$, we have
\begin{equation*} g_2(0) = \sum_{m \equiv 0 \pmod{3}} \frac{(-1)^m}{2m \cdot 2^{2m}} = \sum_{m' > 0} \frac{(-1/4)^{3m'}}{6m'} = -\frac{1}{6} \ln(1+(1/4)^3) = -\frac{1}{6} \ln(65/64).\end{equation*}
When $k \neq 0$, we use the roots of unity. Given $\omega = e^{2 \pi i /3}$, we use the identities
\begin{align*} -\frac{1}{3} \left( \ln \left( 1 - x \right) + \omega^2 \ln(1 - \omega x) + \omega \ln(1 - \omega^2 x) \right) &= x + \frac{x^4}{4} + \frac{x^7}{7} + \dots  \\ -\frac{1}{3} \left( \ln \left( 1 - x \right) + \omega \ln(1 - \omega x) + \omega^2 \ln(1 - \omega^2 x) \right) &= x^2 + \frac{x^5}{5} + \frac{x^8}{8} + \cdots
\end{align*} 
and we get that
\small
\begin{align*} g_2(1) &= -\frac{1}{6} \left( \ln \left( \frac{5}{4} \right) - \ln \left( \sqrt{13} / 4 \right) - \sqrt{3} \arctan \left( \frac{\sqrt{3}}{7} \right) \right), \\
g_2(2) &= \frac{1}{6} \left( \ln \left( \frac{5}{4} \right) - \ln \left( \sqrt{13} / 4 \right) + \sqrt{3} \arctan \left( \frac{\sqrt{3}}{7} \right) \right). \end{align*}

\normalsize
\item[\textbf{(iii)}] $g_3(k)$: We find that
\begin{equation*} g_3(0) = \sum_{m > 0} \frac{1}{2m  \cdot 2^{2m}} = \frac{1}{2} \sum_{m > 0} \frac{(1/4)^m}{m} = -\frac{1}{2} \ln(1- \frac{1}{4}) = -\frac{1}{2} \ln(3/4). \end{equation*}
We have already stated $g_3(1) = g_3(2) = 0$.

\item[\textbf{(iv)}] $g_4(k)$: We rewrite the values as hypergeometric series:
\begin{align*}
g_4(0) &= \sum_{m = 1}^\infty \frac{(-1)^{3m}}{(12m+3) \cdot 2^{6m+1}} = \frac{1}{6} {}_2 F_1 \left( \frac{1}{4}, 1; \frac{5}{4}; -\frac{1}{64} \right) - \frac{1}{6}, \\
g_4(1) &= \sum_{m = 0}^\infty \frac{(-1)^{3m+2}}{(12m+11)2^{6m+5}} = \frac{1}{352} {}_2 F_1 \left( \frac{11}{12}, 1; \frac{23}{12}; -\frac{1}{64} \right), \\
g_4(2) &= \sum_{m = 0}^\infty \frac{(-1)^{3m+1}}{(12m+7)2^{6m+3}} = -\frac{1}{56} {}_2 F_1 \left( \frac{7}{12}, 1; \frac{19}{12}; -\frac{1}{64} \right).
\end{align*}

\item[\textbf{(v)}] $g_5(k)$: We rewrite the values as hypergeometric series: \\
\begin{align*} 
g_5(0) &= \sum_{m \geq 0}^\infty \frac{(-1)^{3m+1}}{(12m+9) \cdot 2^{6m+4}} = - \frac{1}{144} {}_2 F_1 \left( \frac{3}{4}, 1; \frac{7}{4}; -\frac{1}{64} \right), \\
g_5(1) &= \sum_{m \geq 0}^\infty \frac{(-1)^{3m}}{(12m+5) \cdot 2^{6m+2}} = \frac{1}{20} {}_2 F_1 \left( \frac{5}{12}, 1; \frac{17}{12}; -\frac{1}{64} \right), \\
g_5(2) &= \sum_{m \geq 0}^\infty \frac{(-1)^{3m+2}}{(12m+13) \cdot 2^{6m+6}} = 1 - {}_2 F_1 \left( \frac{1}{12}, 1; \frac{13}{12}; -\frac{1}{64} \right).
\end{align*}

\item[\textbf{(vi)}] $g_6(k)$: We find that 
\begin{equation*} g_6(0) = \sum_{m > 0} \frac{1}{2(2m+1) \cdot 2^{2m}} = \frac{1}{2} \sum_{m > 0} \frac{(1/2)^{2m}}{2m+1} = \left( \frac{1}{2} \right) \ln\left(\frac{3}{4}\right)-\ln\left(\frac{1}{2}\right) - \frac{1}{2}. \end{equation*}
We have already stated $g_6(1) = g_6(2) = 0$. 
\end{enumerate}

\vspace{1em}

We now compute $G(k)$ for each $k \in \{ 0, 1, 2 \}$. We write out the details for one case, $G(0)$, and remark that the other two follow analogously. We find that
\begin{align} G(0) &= g_1(0) - g_2(0) + g_3(0) - g_4(0) - g_5(0) + g_6(0) \label{F0} \\
&= \nonumber \frac{1}{3}(1 - \ln (2)) + \frac{1}{6} \ln \left( \frac{65}{64} \right) - \frac{1}{2} \ln \left( \frac{3}{4} \right) - \frac{1}{6} {}_2 F_1 \left( \frac{1}{4}, 1; \frac{5}{4}; -\frac{1}{64} \right) + \frac{1}{6} 
\\&+ \frac{1}{144} {}_2 F_1 \left( \frac{3}{4}, 1; \frac{7}{4}; -\frac{1}{64} \right) + \left( \frac{1}{2} \right) \ln\left(\frac{3}{4}\right)-\ln\left(\frac{1}{2}\right) - \frac{1}{2} \nonumber \\ 
&= \frac{1}{6} \ln \left( \frac{65}{4} \right) - \frac{1}{6} {}_2 F_1 \left( \frac{1}{4}, 1; \frac{5}{4}; -\frac{1}{64} \right) + \frac{1}{144} {}_2 F_1 \left( \frac{3}{4}, 1; \frac{7}{4}; -\frac{1}{64} \right). \nonumber
\end{align}

We use the integral representation \cite[15.6.1]{DLMF}
\begin{equation*} _2F_1(a,b,c;z) = \frac{\Gamma(c)}{\Gamma(b) \Gamma(c-b)} \int_0^1 t^{b-1}(1-t)^{c-b-1} (1-tz)^{-a} dt, \end{equation*} which holds when $\Real(c) > \Real(b) > \Real(0)$ and $\arg(1-z) < \pi$, 
to find
\begin{equation*} _2 F_1 \left( \frac{1}{4}, 1; \frac{5}{4}; \frac{-1}{64} \right) = \frac{\Gamma(5/4)}{\Gamma(1/4)} \int_0^1 \frac{1}{(1-t)^{3/4} (1+\frac{t}{64})^{1/4}} dt = \frac{1}{4} (4+4i)(\cot^{-1}(2+2i) + \coth^{-1}(2+2i)), \end{equation*}
\begin{equation*} _2 F_1 \left( \frac{3}{4}, 1; \frac{7}{4}; \frac{-1}{64} \right) = \frac{\Gamma(7/4)}{\Gamma(3/4)} \int_0^1 \frac{1}{(1-t)^{1/4} (1+\frac{t}{64})^{3/4}} dt = \frac{3}{4} (32-32i)(\cot^{-1}(2+2i) - \coth^{-1}(2+2i)). \end{equation*}
Therefore, the two hypergeometric terms in \eqref{F0} become
\begin{equation*} - \frac{1}{6} _2 F_1 \left( \frac{1}{4}, 1; \frac{5}{4}; \frac{-1}{64} \right) + \frac{1}{144} _2 F_1 \left( \frac{3}{4}, 1; \frac{7}{4}; \frac{-1}{64} \right) = (2i) \cot^{-1}(2+2i) + 2 \coth^{-1}(2+2i) = -\frac{1}{6} \ln \left( \frac{13}{5} \right).  \end{equation*}
Combining terms in \eqref{F0} yields $G(0) = \frac{1}{3} \ln \left( \frac{5}{2} \right)$. We analogously show $G(1) = G(2) = \frac{1}{3} \ln \left( \frac{5}{2} \right)$. 

In \eqref{Flimitexpression}, for all $t'$, we are summing two values of $F$ corresponding to two distinct residues mod 3. Thus, since we have that $G(0) + G(1) = G(0) + G(2) = G(1) + G(2) = \frac{2}{3} \ln \left( \frac{5}{2} \right)$, we obtain that 
\[ \lim_{\substack{t \rightarrow \infty \\ t \equiv t' \pmod{6}}} 2\beta^*_t = \lim_{\substack{t \rightarrow \infty \\ t \equiv t'' \pmod{6}}} 2\beta^*_t \]
for any $t', t''$. Therefore, we conclude that $\lim_{t \rightarrow \infty} 2\beta^*_t = \frac{2}{3} \ln \left( \frac{5}{2} \right).$
This proves Theorem \ref{theorem3}. 

\bibliographystyle{plain}
\bibliography{bib.bib}

\end{document}